\documentclass[a4paper,10pt,reqno]{amsart}
\usepackage{amssymb,latexsym,bbm,amsmath,amsfonts}
\usepackage{amsthm}
\usepackage[mathscr]{eucal}
\usepackage{rotating}
\usepackage{multirow}
\usepackage{slashbox}

\numberwithin{equation}{section}
\newtheorem{theorem}{Theorem}
\newtheorem{lemma}{Lemma}
\newtheorem{proposition}{Proposition}

\theoremstyle{definition}

\newtheorem{example}[theorem]{Example}

\theoremstyle{remark}
\newtheorem{remark}{Remark}

\begin{document}

\title[Asymptotic Behaviour of SDEs with Fading Stochastic Perturbations]
{On the Classification of the Asymptotic Behaviour of Solutions of
Globally Stable Scalar Differential Equations with Respect to State--Independent Stochastic Perturbations}

\author{John A. D. Appleby}
\address{Edgeworth Centre for Financial Mathematics, School of Mathematical
Sciences, Dublin City University, Glasnevin, Dublin 9, Ireland}
\email{john.appleby@dcu.ie} \urladdr{webpages.dcu.ie/\textasciitilde
applebyj}

\author{Jian Cheng}
\address{Edgeworth Centre for Financial Mathematics, School of Mathematical
Sciences, Dublin City University, Glasnevin, Dublin 9, Ireland}
\email{jian.cheng2@mail.dcu.ie}

\author{Alexandra Rodkina}
\address{Department of Mathematics, The University of the West Indies, Mona
Campus, Mona, Kingston 7, Jamaica}
\email{alexandra.rodkina@uwimona.edu.jm}

\thanks{The first and second author gratefully acknowledge Science Foundation Ireland for the support of this research
under the Mathematics Initiative 2007 grant 07/MI/008 ``Edgeworth
Centre for Financial Mathematics''.} 
\date{4 February 2013}

\begin{abstract}
In this paper we characterise the global stability, global
boundedness and recurrence of solutions of a scalar nonlinear
stochastic differential equation. The differential equation is a
perturbed version of a globally stable autonomous equation with
unique equilibrium where the diffusion coefficient is independent of
the state. We give conditions which depend on the rate of decay of
the noise intensity under which solutions either (a) tend to the
equilibrium almost surely, (b) are bounded almost surely but tend to
zero with probability zero, (c) or are recurrent on the real line
almost surely. We also show that no other types of asymptotic
behaviour are possible. Connections between the conditions which
characterise the various classes of long--run behaviour and simple
sufficient conditions are explored, as well as the relationship
between the size of fluctuations and the strength of the mean
reversion and diffusion coefficient, in the case when solutions are
a.s. bounded.
\end{abstract}

\maketitle

\section{Introduction}
In this paper, we characterise the global asymptotic stability of
the unique equilibrium of a scalar deterministic ordinary
differential equation when it is subjected to a stochastic
perturbation independent of the state.

We fix a complete filtered probability space $(\Omega,\mathcal{F},
(\mathcal{F}(t))_{t\geq 0},\mathbb{P})$. Let $B$ be a standard
one--dimensional Brownian motion which is adapted to
$(\mathcal{F}(t))_{t\geq 0}$. We consider the stochastic
differential equation
\begin{equation} \label{eq.sde}
dX(t)=-f(X(t))\,dt + \sigma(t)\,dB(t), \quad t\geq 0; \quad
X(0)=\xi\in\mathbb{R}.
\end{equation}
We suppose that
\begin{equation} \label{eq.fglobalunperturbed}
f\in C(\mathbb{R};\mathbb{R}); \quad xf(x)>0, \quad x\neq 0; \quad
f(0)=0.
\end{equation}
and that $\sigma$ obeys
\begin{equation} \label{eq.sigcns}
\sigma \in C([0,\infty);\mathbb{R}).
\end{equation}
These conditions ensure the existence of a continuous adapted process which
obeys \eqref{eq.sde} on $[0,\infty)$, and we will refer to any such process as a solution.
We do not rule out the existence of more than one process, but part of our analysis will show that
all solutions share the same asymptotic properties. Hypotheses such as local Lipschitz continuity or
monotonicity can be imposed in order to guarantee that there is a unique such solution.

In the case when $\sigma$ is identically zero, it follows under the
hypothesis \eqref{eq.fglobalunperturbed} that any solution $x$ of
\begin{equation}\label{eq.unperturbed}
x'(t)=-f(x(t)), \quad t>0; \quad x(0)=\xi,
\end{equation}
obeys
 \begin{equation} \label{eq.xglobalstable}
\lim_{t\to\infty} x(t;\xi)=0 \text{ for all $\xi\in \mathbb{R}$}.
\end{equation}
Clearly $x(t)=0$ for all $t\geq 0$ if $\xi=0$. The question
naturally arises: if any solution $x$ of \eqref{eq.unperturbed}
obeys \eqref{eq.xglobalstable}, under what conditions on $f$ and
$\sigma$ does any solution $X$ of \eqref{eq.sde} obey
\begin{equation} \label{eq.stochglobalstable}
\lim_{t\to\infty} X(t,\xi)=0, \quad \text{a.s. for each
$\xi\in\mathbb{R}$}.
\end{equation}
The convergence phenomenon captured in \eqref{eq.stochglobalstable}
for any solution of \eqref{eq.sde} is often called almost sure
\emph{global convergence} (or \emph{global stability} for the
solution of \eqref{eq.unperturbed}), because solutions of the
perturbed equation \eqref{eq.sde} converge to the zero equilibrium
solution of the underlying unperturbed equation
\eqref{eq.unperturbed}.

It was shown in Chan and Williams~\cite{ChanWill:1989} that if $f$
is strictly increasing with $f(0)=0$ and $f$ obeys
\begin{equation} \label{eq.ftoinfty}
\lim_{x\to\infty} f(x)=\infty, \quad \lim_{x\to-\infty}
f(x)=-\infty,
\end{equation}
then any solution $X$ of \eqref{eq.sde} obeys
\eqref{eq.stochglobalstable} holds if $\sigma$ obeys
\begin{equation} \label{eq.sigmalogto0}
\lim_{t\to\infty} \sigma^2(t)\log t = 0.
\end{equation}
Moreover, Chan and Williams also proved, if $t\mapsto\sigma^2(t)$ is
decreasing to zero, that if the solution $X$ of \eqref{eq.sde} obeys
\eqref{eq.stochglobalstable}, then $\sigma$ must obey
\eqref{eq.sigmalogto0}. These results were extended to
finite--dimensions by Chan in~\cite{Chan:1989}. The results
in~\cite{ChanWill:1989,Chan:1989} are motivated by problems in
simulated annealing.

In Appleby, Gleeson and Rodkina~\cite{JAJGAR:2009}, the monotonicity
condition on $f$ and \eqref{eq.ftoinfty} were relaxed. It was shown
if $f$ is locally Lipschitz continuous and obeys \eqref{eq.fglobalunperturbed},
and in place of \eqref{eq.ftoinfty} also obeys
\begin{equation} \label{eq.fliminfinfty}
\text{There exists $\phi>0$ such that }\phi:=\liminf_{|x|\to\infty}
|f(x)|,
\end{equation}
then any solution $X$ of \eqref{eq.sde} obeys
\eqref{eq.stochglobalstable} holds if $\sigma$ obeys
\eqref{eq.sigmalogto0}. The converse of Chan and Williams is also
established: if $t\mapsto\sigma^2(t)$ is decreasing, and the solution $X$ of \eqref{eq.sde} obeys \eqref{eq.stochglobalstable},
then $\sigma$ must obey \eqref{eq.sigmalogto0}. Moreover, it was
also shown, without monotonicity on $\sigma$, that if
\begin{equation} \label{eq.sigmalogtoinfty}
\lim_{t\to\infty} \sigma^2(t)\log t = +\infty,
\end{equation}
then the solution $X$ of \eqref{eq.sde} obeys
\begin{equation} \label{eq.stochunstable}
\limsup_{t\to\infty} |X(t,\xi)|=+\infty, \quad \text{a.s. for each
$\xi\in\mathbb{R}$}.
\end{equation}
Furthermore, it was shown that the condition \eqref{eq.sigmalogto0}
could be replaced by the weaker condition
\begin{equation} \label{eq.capSigma0}
\lim_{t\to\infty} \int_0^t e^{-2(t-s)}\sigma^2(s)\,ds \cdot \log_2
\int_0^t \sigma^2(s)e^{2s}\,ds = 0
\end{equation}
and that \eqref{eq.capSigma0} and \eqref{eq.sigmalogto0} are
equivalent when $t\mapsto \sigma^2(t)$ is decreasing. In fact, it
was even shown that if $\sigma^2$ is not monotone decreasing,
$\sigma$ does not have to satisfy \eqref{eq.sigmalogto0} in order
for $X$ to obey \eqref{eq.stochglobalstable}.

In this paper, we improve upon the results in \cite{JAJGAR:2009} and \cite{ChanWill:1989,Chan:1989} in a number of directions. First, we show that neither the Lipschitz continuity of $f$ nor the condition \eqref{eq.fliminfinfty} is needed in order to guarantee that any solution $X$ of \eqref{eq.sde} obeys \eqref{eq.stochglobalstable}. Moreover, we give
necessary and sufficient conditions for the convergence of solutions
which do not require the monotonicity of $\sigma^2$. One of our main
results shows that if $f$ obeys \eqref{eq.fglobalunperturbed}
and $\sigma$ is also continuous, then any solution $X$ of \eqref{eq.sde} obeys \eqref{eq.stochglobalstable} if and only if
\begin{equation} \label{eq.sigmaiffXto0}
S'(\epsilon):= \sum_{n=0}^\infty \sqrt{\int_n^{n+1}\sigma^2(s)\,ds}
\exp\left(-\frac{\epsilon^2}{2}\frac{1}{\int_n^{n+1}\sigma^2(s)\,ds}\right)
<+\infty,\quad \text{for every $\epsilon>0$},
\end{equation}
and it is even shown that if \eqref{eq.sigmaiffXto0} does not hold,
then $\mathbb{P}[X(t)\to 0 \text{ as $t\to\infty$}]=0$ for any
$\xi\in\mathbb{R}$ (Theorem~\ref{theorem.Xiffsigma}). Another significant development from \cite{JAJGAR:2009} and \cite{ChanWill:1989,Chan:1989} is a complete classification of the asymptotic behaviour of \eqref{eq.sde} in terms of
the data, rather than merely satisfactory sufficient conditions. In Theorem~\ref{theorem.Xbounded}, we show that when $f$
obeys \eqref{eq.ftoinfty}, that any solution is either (a) convergent to zero
with probability one, (b) bounded but not convergent to zero, with
probability one, or (c) recurrent on $\mathbb{R}$ with probability
one, according as to whether $S'(\epsilon)$ is always finite,
sometimes finite, or never finite, for $\epsilon>0$. Apart from classifying the asymptotic behaviour, the
novel feature here is that bounded but non--convergent solutions are examined.

Although the condition \eqref{eq.sigmaiffXto0} is necessary and
sufficient for $X$ to obey \eqref{eq.stochglobalstable}, it may
prove to be a little unwieldy for use in some situations. For this
reason we deduce some sharp sufficient conditions for $X$ to obey
\eqref{eq.stochglobalstable}. If $f$ obeys
\eqref{eq.fglobalunperturbed} 
and $\sigma$ is continuous and obeys \eqref{eq.capSigma0}, then any solution $X$ of \eqref{eq.sde} obeys \eqref{eq.stochglobalstable} (Theorem~\ref{theorem.sufficientforx}). In the spirit of
Theorem~\ref{theorem.Xiffsigma}, we also establish converse results
in the case when $\sigma^2$ is monotone
(Theorem~\ref{theorem.eqvt}), and demonstrate that the condition
\eqref{eq.capSigma0} is hard to relax if we require $X$ to obey
\eqref{eq.stochglobalstable}. The relationship between the
conditions which characterise the asymptotic behaviour, and which
involve $S'(\epsilon)$, and sufficient conditions are explored in
several results, notably in Proposition~\ref{prop.sigmasqlogtS}
and~\ref{prop.thetaSigma}. Also, in the case when solutions are
bounded, we analyse the relationships between the deterministic
bounds on solutions and the drift and diffusion coefficients. In
particular, in Propositions~\ref{prop.underline},\ref{prop.upperboundX}
and~\ref{prop.overline}, we demonstrate the bounds on any solution
increase with greater noise intensity, and with weaker mean
reversion.

These results are proven by showing that the stability of
\eqref{eq.sde} is equivalent to the asymptotic stability of a
process $Y$ which is the solution of an affine SDE with the same
diffusion coefficient $\sigma$ (Proposition~\ref{prop.yimpliesx}, especially part (A)). A
classification of the asymptotic behaviour $Y$ has already been
achieved in \cite{AppCheRod:2010a, AppCheRod:dres}, and the relevant results of
\cite{AppCheRod:2010a} are listed in Section~\ref{sec.linear}. The proof of part (a) of Proposition~\ref{prop.yimpliesx} is given under the additional condition that $\sigma\not \in L^2(0,\infty)$; the
case when $\sigma\in L^2(0,\infty)$ is easier, uses different
methods, and is dealt with separately in Theorem~\ref{th.sigL2}.
Essentially, in the case when $\sigma\not\in L^2(0,\infty)$ the
recurrence of one--dimensional standard Brownian motion forces solutions
to return to an arbitrarily small neighbourhood of the
origin infinitely often. Then, if the noise fades sufficiently
quickly so that the affine SDE is convergent to zero, the difference
$Z:=X-Y$ obeys a perturbed version of the ordinary differential
equation \eqref{eq.unperturbed} where the perturbation fades to zero
asymptotically, and by virtue of the recurrence property, there
exist arbitrarily large time when $Z$ is arbitrarily close to zero.
By considering an initial value problem for $Z$ starting at such
times, deterministic methods can then be used to show that $Z$ tends
to zero, and hence that $X$ tends to zero. A similar method is
employed in Theorem~\ref{theorem.Xbounded} to establish an upper
bound on $|X|$ when $Y$ is bounded, but does not tend to zero.
Establishing that solutions of \eqref{eq.sde} is unbounded, or
obeys certain lower bounds, is generally achieved by writing a
variation of constants formula for $X$ in terms on $Y$, and then
using the known asymptotic behaviour of $Y$ to force a
contradiction.

Although many parts of the analysis in this paper apply to
finite--dimensional equations, we do not pursue this question here.
The major reason for doing has already been mentioned above. An
important ingredient in establishing the asymptotic stability in the
case when $\sigma\not \in L^2(0,\infty)$ is the fact that process
$M$ defined by $M(t)=\int_0^t \sigma(s)\,dB(s)$ can be considered as
a one--dimensional Brownian motion on $[0,\infty)$, and therefore
returns to the origin infinitely often. For the one--dimensional
SDE, this causes the solution to return to an arbitrarily small
neighbourhood of the equilibrium infinitely often, and once other
stochastic terms fade, ensures convergence to the equilibrium.

Once we turn to consider higher dimensional equations, the
corresponding stochastic process $M(t)=\int_0^t \sigma(s)\,dB(s)$ in
higher dimensions may start to inherit some properties of
finite--dimensional Brownian motion, for dimension greater than or
equal to three. However, this means that $M$ can be transient,
obeying \[\lim_{t\to \infty} |M(t)|=+\infty, \quad\text{a.s.}\] In
this situation, therefore, one can no longer use the scalar method
of proof to show that the solution of the finite dimensional SDE
returns to an arbitrarily small neighbourhood of the equilibrium
infinitely often. At this moment, it is unclear to the authors
whether this is merely a technical problem, or if it presages
different asymptotic behaviour, with solutions losing stability more
readily than in the scalar case.

Other interesting questions which can be attacked by means of the
methods in this paper include an analysis of local stability, where
there are a finite number of equilibria of the underlying
deterministic dynamical system \eqref{eq.unperturbed}. Some work in
this direction has been conducted in a discrete--time setting in
\cite{JAGBAR:2008}. Numerical methods under monotonicity methods
have been studied in~\cite{JAJCAR:2010}. We expect that the sharper
information on the asymptotic behaviour of the linear SDE will lead
to improved results for the corresponding nonlinear equations.

Section 2 deals with preliminary results, including the proof that
solutions of \eqref{eq.sde} exist. Results for an auxiliary affine
SDE, proven in \cite{JAJCAR:2010}, are recapitulated in Section 3,
along with some new results for the
stability of affine equations. 
Section 4 considers general results, including the classification of
the almost sure behaviour of solutions under the additional
assumption \eqref{eq.fliminfinfty} on $f$. Section 5 considers the
characterisation of asymptotic stability using only the assumption
\eqref{eq.fglobalunperturbed}. Proofs of many results are deferred
to the end of the paper, and these proofs are presented in Sections
\ref{sec.proofsexist}, \ref{sec.prelimproof}, \ref{sec.proofs1},
\ref{sec:s9} and \ref{sec:s10}.

\section{Preliminaries}

\subsection{Notation}
In advance of stating and discussing our main results, we introduce
some standard notation. We denote the maximum of the real numbers
$x$ and $y$ by $x\vee y$ and the minimum of $x$ and $y$ by $x\wedge
y$. Let $C(I;J)$ denote the space of continuous functions $f:I\to J$
where $I$ and $J$ are intervals contained in $\mathbb{R}$. We denote
by $L^1(0,\infty)$ the space of Lebesgue integrable functions
$f:[0,\infty)\to\mathbb{R}$ such that $\int_0^\infty |f(s)|\,ds < +
\infty$.

\subsection{Remarks on existence and uniqueness of solutions of
\eqref{eq.sde}} \label{sec:exist} There is an extensive theory
regarding the existence and uniqueness of solutions of stochastic
differential equations under a variety of regularity conditions on
the drift and diffusion coefficients. Perhaps the most commonly
quoted conditions which ensure the existence of a strong local
solution are the Lipschitz continuity of the drift and diffusion
coefficients. However, in this paper, we would like to establish our
asymptotic results under weaker hypotheses on $f$. We do not concern
ourselves greatly with relaxing conditions on $\sigma$, because
$\sigma$ being continuous proves sufficient to ensure the existence
of solutions in many cases.

The existence of a unique solution of
\begin{equation}
\label{eq.sde2} dX(t)=f(X(t))\,dt + \sigma(t,X(t))\,dB(t)
\end{equation}
can be asserted in the case when $|\sigma(t,x)|\geq c>0$ for some
$c>0$ for all $(t,x)$ and $f$ being bounded, so no continuity
assumption is required on $f$. However, assuming such a lower bound
on $\sigma$ would not natural in the context of this paper: for
asymptotic stability results, we would typically require that
$\liminf_{t\to\infty} \sigma^2(t)=0$. Moreover, $f$ being bounded
excludes the important category of strongly mean--reverting
functions $f$ that have been investigated for this stability problem
in \cite{ChanWill:1989} and \cite{JAJGAR:2009}.

One of the aims of this paper and of \cite{JAJGAR:2009} is to relax
monotonicity assumptions on $\sigma$ which are required in
\cite{ChanWill:1989}. Therefore, although we are often interested in
functions $\sigma$ which tend to zero in some sense, we do not want
to exclude the cases when $\sigma(t)=0$ for all $t$ in a given
interval (or indeed union of intervals). Our analysis will show that
in these cases, the behaviour of $\sigma$ on the intervals where it
is nontrivial can give rise to solutions of \eqref{eq.sde} obeying
\eqref{eq.stochglobalstable} or \eqref{eq.stochunstable}. However,
on those time intervals $I$ for which $\sigma$ is zero, the process
$X$ obeys the differential equation
\[
X'(t)=-f(X(t)), \quad t\in I
\]
where $X(\inf I)$ is a random variable. On such an interval, it is
conceivable that a lack of regularity in $f$ could give rise to
multiple solutions of the ordinary equation (and hence the SDE
\eqref{eq.sde}), so our most general existence results which make
assertions about the existence of solutions (but say nothing about
unicity of solutions), and which use the weakest hypotheses on $f$
that we impose in this work, do not appear to be especially
conservative.

For these reasons, we prove that there is a continuous and adapted
process which obeys \eqref{eq.sde} by using very elementary methods,
rather than by appealing to a result from the substantial body of
sophisticated theory concerning the existence of solutions of
\eqref{eq.sde2}. Of course, these methods are of very limited
utility in establishing existence and uniqueness for more general
equations of the form \eqref{eq.sde2}: our method of proof works
because the diffusion coefficient is independent of the state. In
fact, our method of proof gives a weaker conclusion for the
existence of solutions of \eqref{eq.sde2} in the case when $f$ is
bounded and $\sigma(t,x)=\sigma(t)$ for all $(t,x)$ and
$|\sigma(t)|\geq c>0$. Our result states that there is a solution of
\eqref{eq.sde} (or \eqref{eq.sde2}) while existing results guarantee the existence of a unique solution under the assumption 
that $f$ be continuous. Despite these
general limitations, however, our proof does ensure existence of
solutions for all the problems that are of concern in this paper,
while existing results cannot always be applied without the
imposition of additional hypotheses.

When $f$ obeys \eqref{eq.fglobalunperturbed} and $\sigma$ obeys
\eqref{eq.sigcns}, we now demonstrate that there exists a continuous
and adapted process $X$ which satisfies \eqref{eq.sde}. The
existence of a local solution is ensured by the continuity of $f$
and $\sigma$, while the fact that any such solution is well--defined
for all time follows from the mean--reverting condition $xf(x)>0$
for $x\neq 0$ which is part of \eqref{eq.fglobalunperturbed}. In the
paper, the spirit of our approach is to show that \emph{any}
solution of \eqref{eq.sde} has the stated asymptotic properties,
even though multiple solutions exist, without paying particular
concern as to whether solutions are unique.


\begin{proposition}  \label{prop.exist}
Suppose that $f$ obeys \eqref{eq.fglobalunperturbed} and $\sigma$ obeys \eqref{eq.sigcns}. Then there exists a
continuous adapted process $X$ which obeys \eqref{eq.sde} on $[0,\infty)$, a.s.
\end{proposition}
The proof is postponed to Section~\ref{sec.proofsexist}. In order to
ensure that solutions of \eqref{eq.sde} are unique, it is often
necessary to impose additional regularity properties on $f$. One
common and mild assumption which ensures uniqueness is that
\begin{equation} \label{eq.floclip}
\text{$f$ is locally Lipschitz continuous}.
\end{equation}
See e.g., \cite{Mao1}. Another assumption which guarantees the
uniqueness of the solution is that the drift coefficient $-f$ obeys
a one--sided Lipschitz condition. More precisely, imposing such an
assumption on $f$ implies
\begin{equation} \label{eq.fonesidedlip}
\text{There exists $K\in \mathbb{R}$ such that }
(f(x)-f(y))(x-y)\geq -K|x-y|^2 \quad \text{for all
$x,y\in\mathbb{R}$}.
\end{equation}
 It is to be noted that if $f$ is non--decreasing, it obeys \eqref{eq.fonesidedlip}, because the righthand side is non--negative,
 and we can choose $K=0$. Since non--decreasing functions do not have to be Lipschitz continuous, we see that in general \eqref{eq.fonesidedlip} does not imply \eqref{eq.floclip}, so these additional assumptions can be used to cover different situations.
\begin{proposition}\label{prop.existonesided}
Suppose that $f$ obeys \eqref{eq.fglobalunperturbed} and
\eqref{eq.fonesidedlip} and that $\sigma$ obeys \eqref{eq.sigcns}.
Then there exists a unique continuous adapted process $X$ which
obeys \eqref{eq.sde} on $[0,\infty)$ a.s.
\end{proposition}
Again the proof is deferred to Section~\ref{sec.proofsexist}.

In the proof of Proposition~\ref{prop.exist}, and elsewhere
throughout the paper, it is helpful to introduce the following
processes and notation. Consider the affine stochastic differential
equation
\begin{equation} \label{eq.linsde}
dY(t)=-Y(t)\,dt+\sigma(t)\,dB(t),\quad t\geq 0; \quad Y(0)=0.
\end{equation}
Since $\sigma$ is continuous, there is a unique continuous adapted
process which obeys \eqref{eq.linsde}, and we identify the a.s.
event $\Omega_Y$ on which this solution is defined:
\begin{multline}
\Omega_Y=\big\{\omega\in \Omega: \text{there is a unique continuous adapted process $Y$} \\
\label{def.OmegaY}
 \text{for which the realisation $Y(\cdot,\omega)$ obeys \eqref{eq.linsde}}\bigr\}.
\end{multline}
It is also helpful throughout the paper to identify the event
$\Omega_X$ on which the continuous adapted process $X$ obeys
\eqref{eq.sde}, so we therefore define
\begin{multline}  \label{def.OmegaX}
\Omega_X=\bigl\{\omega\in \Omega: \text{the continuous adapted process $X$} \\
\text{is such that the realisation $X(\cdot,\omega)$ obeys \eqref{eq.sde}}\bigr\}.
\end{multline}
By virtue of Proposition~\ref{prop.exist}, $\Omega_X$ is an almost sure event.

\subsection{Preliminary asymptotic results}
We first consider hypotheses on the data i.e., on $\sigma$ under
which any solution $X$ of \eqref{eq.sde} obeys
\eqref{eq.stochglobalstable}. We note that when $\sigma\in
L^2(0,\infty)$, we have $X$ obeys \eqref{eq.stochglobalstable}.
However, we cannot apply directly the semimartingale convergence
theorem of Lipster--Shiryaev (see e.g., \cite[Theorem 7,
p.139]{LipShir:1989}) to the non--negative semimartingale $X^2$,
because it is not guaranteed that $\mathbb{E}[X^2(t)]<+\infty$ for
all $t\geq 0$. The proof of the following theorem, which is deferred
to the next section, uses the ideas of \cite[Theorem 7,
p.139]{LipShir:1989} heavily, however.
\begin{theorem} \label{th.sigL2}
Suppose that $f$ satisfies \eqref{eq.fglobalunperturbed}.
Suppose that $\sigma$ obeys \eqref{eq.sigcns}
and $\sigma\in L^2(0,\infty)$. If $X$ is any solution of
\eqref{eq.sde}, then $X$ obeys \eqref{eq.stochglobalstable}.
\end{theorem}
The proof is relegated to Section~\ref{sec.sigmal2proof}. Our next
result shows that if, on the contrary, $\sigma\not\in
L^2(0,\infty)$, we can only guarantee that $X$ visits a
neighbourhood of the equilibrium infinitely often.
\begin{theorem} \label{lemma.Xliminf}
Suppose that $f$ obeys \eqref{eq.fglobalunperturbed},
and that $\sigma$ obeys \eqref{eq.sigcns} and
$\sigma\not\in L^2(0,\infty)$. Then any solution $X$ of
\eqref{eq.sde} obeys $\liminf_{t\to\infty} |X(t)|=0$ a.s.
\end{theorem}
Again the proof is postponed to Section~\ref{sec.sigmal2proof}.
\section{Linear Equation}  \label{sec.linear}
We start by recalling results concerning the asymptotic behaviour of
the related affine stochastic differential equation \eqref{eq.linsde}. These were
presented in Appleby, Cheng, and Rodkina~\cite{AppCheRod:2010a}.
There is a unique continuous adapted processes which obeys \eqref{eq.linsde}.
We note that $Y$ can be written in the form
\begin{equation} \label{eq.Yform}
Y(t)=e^{-t}\int_0^t e^s\sigma(s)\,dB(s), \quad t\geq 0.
\end{equation}
Let $\Phi:\mathbb{R}\to [0,1]$ be the distribution function of a
standard normal random variable, so that
\begin{equation}\label{def.phi}
\Phi(x)=\frac{1}{\sqrt{2\pi}}\int_{-\infty}^x
e^{-\frac{1}{2}u^2}\,du, \quad x\in\mathbb{R}.
\end{equation}
We interpret $\Phi(-\infty)=0$ and $\Phi(\infty)=1$. Define the
sequence $\theta:\mathbb{N}\to[0,\infty)$ such that
\begin{equation} \label{def.theta}
\theta^2(n)=\int_n^{n+1} \sigma^2(s)\,ds.
\end{equation}
Let $\epsilon>0$ and consider the sum
\begin{equation} \label{def.Sepsilon}
S(\epsilon)=\sum_{n=0}^\infty
\left\{1-\Phi\left(\frac{\epsilon}{\theta(n)}\right)\right\}.
\end{equation}
This summation is difficult to evaluate directly, because $\Phi$ is
not known in closed form. However, it can be shown that
$S(\epsilon)$ is finite or infinite according as to whether the sum
\begin{equation}\label{def.s'epsilon}
S'(\epsilon):=\sum_{n=0}^\infty
\theta(n)\exp\left(-\frac{\epsilon^2}{2}\frac{1}{\theta^2(n)}\right)
\end{equation}
is finite or infinite, where we interpret the summand to be zero in
the case where $\theta(n)=0$.
\begin{lemma} \label{lemma.sepss'eps}
$S(\epsilon)$ given by \eqref{def.Sepsilon} is finite if and only if
$S'(\epsilon)$ given by \eqref{def.s'epsilon} is finite.
\end{lemma}
\begin{proof}
We note by e.g., \cite[Problem 2.9.22]{K&S},
\begin{equation} \label{eq.millsasy}
\lim_{x\to\infty}\frac{1-\Phi(x)}{x^{-1}e^{-x^2/2}}=\frac{1}{\sqrt{2\pi}}.
\end{equation}
If $S(\epsilon)$ is finite, then $1-\Phi(\epsilon/\theta(n))\to 0$
as $n\to\infty$. This implies $\epsilon/\theta(n)\to\infty$ as
$n\to\infty$. Therefore by \eqref{eq.millsasy}, we have
\begin{equation} \label{eq.millsthetan}
\lim_{n\to\infty}\frac{1-\Phi(\epsilon/\theta(n))}{\theta(n)/\epsilon\cdot
\exp(-\epsilon^2/\{2\theta^2(n)\})}=\frac{1}{\sqrt{2\pi}}.
\end{equation}
Since $(1-\Phi(\epsilon/\theta(n)))_{n\geq 1}$ is summable, it
therefore follows that the sequence
\[(\theta(n)/\epsilon\cdot
\exp(-\epsilon^2/\{2\theta^2(n)\}))_{n\geq 1}
\] is summable, so $S'(\epsilon)$ is finite, by definition.

On the other hand, if $S'(\epsilon)$ is finite, and we define
$\phi:[0,\infty)\to \mathbb{R}$ by
\[
\phi(x)=\left\{ \begin{array}{cc} x \exp(-1/(2x^2)), & x>0, \\
0, & x=0,
\end{array}
\right.
\]
we have that $(\phi(\theta(n)/\epsilon))_{n\geq 1}$ is summable.
Therefore $\phi(\theta(n)/\epsilon)\to 0$ as $n\to\infty$. Then, as
$\phi$ is continuous and increasing on $[0,\infty)$, we have that
$\theta(n)/\epsilon\to 0$ as $n\to\infty$, or $\epsilon/\theta(n)\to
\infty$ as $n\to\infty$. Therefore \eqref{eq.millsthetan} holds, and
thus $(1-\Phi(\epsilon/\theta(n)))_{n\geq 1}$ is summable, which
implies that $S(\epsilon)$ is finite, as required.
\end{proof}

Clearly, there are three possibilities; either (A) $S'(\epsilon)$ is
finite for all $\epsilon>0$; (B) $S'(\epsilon)$ is infinite for all
$\epsilon>0$ or (C) $S'(\epsilon)$ is finite for some $\epsilon>0$
and infinite for others. In the last case, we notice that as $\Phi$
is increasing, we have that $\epsilon \mapsto S'(\epsilon)$ is
non--increasing, there must exist an $\epsilon'>0$ such that
$S'(\epsilon)$ is finite for all $\epsilon>\epsilon'$ and
$S'(\epsilon)=+\infty$ for all $\epsilon<\epsilon'$.

By virtue of this observation, it can be seen that the following
result characterises the asymptotic behaviour of $Y$. Roughly
speaking, if $S'(\epsilon)$ is always finite, $Y$ tends to zero with
probability one; if $S'(\epsilon)$ can be finite or infinite, it
fluctuates asymptotically within interval of the form $[-c,c]$,
while if $S'(\epsilon)$ is always infinite, $Y$ is recurrent on
$\mathbb{R}$.
\begin{theorem}  \label{theorem.Ybounded}
Suppose that $\sigma$ obeys \eqref{eq.sigcns} and that $Y$ is the
unique continuous adapted process which obeys \eqref{eq.linsde}. Let
$\theta$ be defined by \eqref{def.theta} and $S'(\cdot)$ by
\eqref{def.s'epsilon}.
\begin{itemize}
\item[(A)] If $\theta$ is such that
\begin{equation} \label{eq.thetastable}
\text{$S'(\epsilon)$ is finite for all $\epsilon>0$},
\end{equation}
then
\begin{equation} \label{eq.Ytto0}
\lim_{t\to\infty} Y(t)=0, \quad\text{a.s.}
\end{equation}
\item[(B)] If $\theta$ is such that there exists $\epsilon'>0$ such that
\begin{subequations} \label{eq.thetabounded}
\begin{align} \label{eq.theta1}
\text{$S'(\epsilon)$ is finite for all $\epsilon>\epsilon'$},\\
\label{eq.theta2} \text{$S'(\epsilon)=+\infty$ for all
$\epsilon<\epsilon'$},
\end{align}
\end{subequations}
then the event $\Omega_1$ defined by
\begin{equation} \label{def.Omega1}
\Omega_1:=\{\omega\in \Omega_Y:0<\limsup_{t\to \infty}|Y(t, \omega)|<+\infty\}.
\end{equation}
is almost sure and there exist deterministic $0<\underline{Y}\leq \overline{Y}<+\infty$ defined by
\begin{align} \label{def.underY}
\underline{Y}&:=\inf_{\omega\in \Omega_1} \limsup_{t\to \infty}|Y(t,\omega)|>0, \\
\label{def.overY}
\overline{Y}&:=\sup_{\omega\in \Omega_1} \limsup_{t\to \infty}|Y(t,\omega)|>0.
\end{align}
%
\item[(C)] If $\theta$ is such that
\begin{equation} \label{eq.thetaunstable}
\text{$S'(\epsilon)=+\infty$ for all $\epsilon>0$},
\end{equation}
then
\begin{equation} \label{eq.Ytunstable}
\limsup_{t\to\infty} |Y(t)|=+\infty, \quad \liminf_{t\to\infty}
|Y(t)|=0, \quad \text{a.s.}
\end{equation}
\end{itemize}
\end{theorem}
In Theorem~\ref{theorem.Ybounded}, no monotonicity conditions are
imposed on $\sigma$. The form of Theorem~\ref{theorem.Ybounded} is
inspired by those of \cite[Theorem 1]{ChanWill:1989} and
\cite[Theorem 6, Corollary 7]{JAARMR:2009}.
\begin{remark} \label{remark.estc1c2}
The existence of deterministic bounds on $|Y|$ in \eqref{def.underY} and \eqref{def.overY} in part (B) was established as part of Theorem
1 in~\cite{AppCheRod:2010a}. Moreover, it was established as part of
the proof that explicit bounds on $\overline{Y}$ and $\underline{Y}$ can be
given in terms of the critical value of $\epsilon=\epsilon'$ in
\eqref{eq.thetabounded}. The estimates given by the analysis in
\cite{AppCheRod:2010a} are
\begin{equation} \label{eq.estYbarYunderbar}
\underline{Y}\geq \underline{y}:=\frac{e^{-1}}{1+e^{-1}}\epsilon', \quad \overline{Y}\leq
\left(\frac{1}{1-e^{-1}} + e\right)\epsilon'=:\overline{y}.
\end{equation}
Hence we have $0.2689\epsilon'\leq \limsup_{t\to\infty} |Y(t)|\leq
4.3003\epsilon'$, a.s.

It remains an open question as to whether in general the explicit bounds
$\overline{y}$ and $\underline{y}$  on $\overline{Y}$ and $\underline{Y}$ can be improved. In part of
Theorem~\ref{theorem.sufficient} in which case (B) holds, it can be shown by an independent
argument that $\underline{y}= \overline{y}=\epsilon'$ and therefore that
$\overline{Y}=\underline{Y}=\epsilon'$.
%
\end{remark}
\begin{remark} \label{rem.sigL2}
If $\sigma$ obeys \eqref{eq.sigcns} and $\sigma\in L^2(0,\infty)$,
and $Y$ is the solution of \eqref{eq.linsde}, then $Y$ obeys
 $\lim_{t\to\infty} Y(t)=0$ a.s. by Theorem~\ref{th.sigL2}.
Moreover, if $\sigma\in L^2(0,\infty)$, then $\sigma$ obeys
\eqref{eq.thetastable}. If $\sigma$ obeys either
\eqref{eq.thetabounded} or \eqref{eq.thetaunstable}, then
$\sigma\not\in L^2(0,\infty)$.
\end{remark}
The condition that $S'(\epsilon)$ is finite or infinite can be
difficult to check. However, in the case when
\begin{equation} \label{eq.sigsqlogt}
\text{There exists $L\in [0,\infty]$ such that } L=\lim_{t\to\infty}
\sigma^2(t)\log t,
\end{equation}
each of the conditions \eqref{eq.thetastable},
\eqref{eq.thetabounded} and \eqref{eq.thetaunstable} is possible
according as to whether the limit $L$ is zero, non--zero and finite,
or infinite. In this case therefore, the asymptotic behaviour of any solution of \eqref{eq.sde} can be classified completely.
\begin{proposition} \label{prop.sigmasqlogtS}
Suppose that $\sigma\in C([0,\infty);\mathbb{R})$ obeys
\eqref{eq.sigsqlogt} and that $S'(\cdot)$ is defined by
\eqref{def.s'epsilon}.
\begin{itemize}
\item[(A)] If $L=0$, then $S'$ obeys \eqref{eq.thetastable}.
\item[(B)] If $L\in (0,\infty)$, then $S'$ obeys \eqref{eq.thetabounded}.
\item[(C)] If $L=\infty$, then $S'$ obeys \eqref{eq.thetaunstable}.
\end{itemize}
\end{proposition}
Scrutiny of the proof reveals that we can replace the condition
\eqref{eq.sigsqlogt} with the weaker condition
 \begin{equation} \label{eq.thetasqlogt}
\text{There exists $L\in [0,\infty]$ such that } L=\lim_{n\to\infty}
\theta^2(n)\log n,
\end{equation}
and still obtain the same trichotomy in Proposition~\ref{prop.sigmasqlogtS}. The proof of Proposition~\ref{prop.sigmasqlogtS} is postponed to Section~\ref{sec.proofs1}.

The conditions of Theorem~\ref{theorem.Ybounded} can be quite
difficult to check in practice. In \cite{AppCheRod:2010a},
easily--checked sufficient conditions on $\sigma$ for which $Y$ is
bounded, stable or unstable, are developed. These results are extended slightly here, and will
also be used to analyse the nonlinear equation \eqref{eq.sde}. For this reason, they are stated afresh here.

In the case when $\sigma\in L^2(0,\infty)$ we have that $Y$ tends to
zero. Therefore, we confine attention to the case where
$\sigma\not\in L^2(0,\infty)$. In this case, we can define a number
$T>0$ such that $\int_0^t e^{2s}\sigma^2(s)\,ds>e^e$ for $t>T$ and
so one can define a function $\Sigma:[T,\infty)\to [0,\infty)$ by
\begin{equation} \label{def.Sigma}
\Sigma(t)=\left( \int_0^t e^{-2(t-s)}\sigma^2(s)\,ds \right)^{1/2}
\left(\log t  \right)^{1/2}, \quad t\geq T.
\end{equation}
Our main result in this direction can now be stated. Apart from part
(C) it appears in \cite[Theorem 3.2]{AppCheRod:2010a}.
\begin{theorem} \label{theorem.sufficient}
Suppose that $\sigma$ obeys \eqref{eq.sigcns} and that $Y$ is the
unique continuous adapted process which obeys \eqref{eq.linsde}. Let
$\Sigma$ be given by \eqref{def.Sigma}.
\begin{itemize}
\item[(A)] If $\lim_{t\to \infty} \Sigma^2(t)=0$
then
\begin{equation} \label{eq.Y0}
\lim_{t\to\infty} Y(t)=0, \quad\text{a.s.}
\end{equation}
\item[(B)] If $\liminf_{t\to \infty} \Sigma^{2}(t)=L<+\infty$ then
\begin{equation} \label{eq.Yfiniteliminf}
\limsup_{t\to\infty} |Y(t)|\geq \sqrt{2L},\quad \text{a.s.}
\end{equation}
\item[(C)] If $\limsup_{t\to \infty} \Sigma^{2}(t)=L<+\infty$ then
\begin{equation} \label{eq.Yfinitelimsup}
\limsup_{t\to\infty} |Y(t)|\leq \sqrt{2L},\quad \text{a.s.}
\end{equation}
\item[(D)] If $\lim_{t\to \infty} \Sigma^{2}(t)=L<+\infty$ then
\begin{equation} \label{eq.Yfinite}
\limsup_{t\to\infty} |Y(t)|=\sqrt{2L},\quad \text{a.s.}
\end{equation}
\item[(E)] If $\lim_{t\to \infty} \Sigma^{2}(t)=+\infty$
then
\begin{equation} \label{eq.Yinfinite}
\limsup_{t\to\infty} |Y(t)|=+\infty, \quad \text{a.s.}
\end{equation}
\end{itemize}
\end{theorem}
The proof of part (C) uses the methods of \cite[Theorem
3.2]{AppCheRod:2010a}, so is not given. It is now clear that part
(D) is merely a corollary of parts (B) and (C). Parts (A) and (E)
may also be thought of as limiting cases of part (D) as $L\to 0$ and
$L\to\infty$, respectively. We note that when $\sigma$ obeys
\eqref{eq.sigsqlogt}, then $\Sigma^2(t)\to L$ as $t\to\infty$, so
that in part (D), we have from the proof of part (B) of
Proposition~\ref{prop.sigmasqlogtS} that $S'$ obeys
\eqref{eq.thetabounded} with $\epsilon'=\sqrt{2L}$ and by
\eqref{eq.Yfinite}, that
$\overline{Y}=\underline{Y}=\sqrt{2L}=\epsilon'$ in \eqref{def.underY} and \eqref{def.overY}. This strengthens the general estimates given
on  $\overline{Y}$ and $\underline{Y}$ in
\eqref{eq.estYbarYunderbar}.

Theorem~\ref{theorem.Ybounded} gives necessary and sufficient
conditions in terms of the sequence $\theta$ for $Y$ to exhibit
certain types of asymptotic behaviour, while
Theorem~\ref{theorem.sufficient} gives sufficient conditions in
terms of the function $\Sigma$. In the next result, we explore the
relationship between $\Sigma$ and $\theta$, and the conditions in
Theorems~\ref{theorem.Ybounded} and \ref{theorem.sufficient}. One
consequence of this analysis is to give simpler sufficient
conditions equivalent to those in part (A) of
Theorem~\ref{theorem.sufficient} under which $Y$ tends to zero.
\begin{proposition}  \label{prop.thetaSigma}
Suppose that $\Sigma$ is given by \eqref{def.Sigma},  that $\theta$
is given by \eqref{def.theta}, and that $\Theta$ is given by
\begin{equation} \label{def.capTheta}
\Theta^2(n)= \sum_{j=0}^{n-1} e^{-2(n-j)} \theta^2(j), \quad n\geq
1.
\end{equation}
\begin{itemize}
\item[(i)]  The following statements are equivalent:
\begin{enumerate}
\item[(A)] $\lim_{t\to\infty} \Sigma^2(t)=0$;
\item[(B)] $\lim_{n\to\infty} \Sigma^2(n)=0$;
\item[(C)] $\lim_{n\to\infty} \theta^2(n)\log n =0$.
\end{enumerate}
Moreover, all imply that $S'(\epsilon)<+\infty$ for all
$\epsilon>0$.
\item[(ii)] The following statements are equivalent:
\begin{enumerate}
\item[(A)] $\limsup_{t\to\infty} \Sigma^2(t)\in (0,\infty)$;
\item[(B)] $\limsup_{n\to\infty} \Sigma^2(n)\in (0,\infty)$;
\item[(C)] $\limsup_{n\to\infty} \theta^2(n)\log n \in (0,\infty)$.
\end{enumerate}
Moreover, all imply that there exists $\epsilon'>0$ such that
$S'(\epsilon)<+\infty$ for all $\epsilon>\epsilon'$.
\item[(iii)] The following statements are equivalent:
\begin{enumerate}
\item[(A)] $\liminf_{t\to\infty} \Sigma^2(t)\in (0,\infty)$;
\item[(B)] $\liminf_{n\to\infty} \Sigma^2(n)\in (0,\infty)$;
\item[(C)] $\liminf_{n\to\infty} \Theta^2(n)\log n \in (0,\infty)$.
\end{enumerate}
Moreover,  all imply that $\liminf_{n\to\infty} \theta^2(n)\log n\in
[0,\infty)$.
\item[(iv)] The following statements are equivalent:
\begin{enumerate}
\item[(A)] $\lim_{t\to\infty} \Sigma^2(t)=+\infty$;
\item[(B)] $\lim_{n\to\infty} \Sigma^2(n)=+\infty$;
\item[(C)] $\lim_{n\to\infty} \Theta^2(n)\log n =+\infty$.
\end{enumerate}
Moreover, all imply that $\limsup_{n\to\infty} \theta^2(n)\log n
=\infty$.
\end{itemize}
\end{proposition}
Once again, the proof is relegated to Section~\ref{sec.proofs1}.
%

\section{Nonlinear Equation}
In this section we explore the asymptotic behaviour of the nonlinear
differential equation \eqref{eq.sde}. In the first part of this
section, we establish a connection between the solution of
\eqref{eq.linsde} and solutions of \eqref{eq.sde}. This enables us to state the
main results of the paper, which appear, together with
interpretation and examples, in the second part of this section.
\subsection{Connection between the linear and nonlinear equation}
In our first result, we show that knowledge of the pathwise
asymptotic behaviour of $Y(t)$ as $t\to\infty$ enables us to infer a
great deal about the asymptotic behaviour of $X(t)$ as $t\to\infty$.
Indeed, we show in broad terms that $X$ inherits the asymptotic
behaviour exhibited by $Y$, when $f$ obeys
\eqref{eq.fglobalunperturbed}.
\begin{proposition} \label{prop.yimpliesx}
Suppose that $f$ satisfies \eqref{eq.fglobalunperturbed} and
that $\sigma$ obeys \eqref{eq.sigcns}.
Let $X$ be any solution of \eqref{eq.sde}, and $Y$ the solution of
\eqref{eq.linsde}, and suppose that the a.s. events $\Omega_X$ and $\Omega_Y$ are defined as in \eqref{def.OmegaX} and \eqref{def.OmegaY} respectively.
\begin{itemize}
\item[(A)] Suppose that there is an a.s. event  defined by
\[
\{\omega\in \Omega_Y:\lim_{t\to \infty}|Y(t, \omega)|=0\}.
\]
Then $\lim_{t\to \infty}X(t)=0$ a.s.
\item[(B)]
Suppose that the event $\Omega_1$  defined by \eqref{def.Omega1} is almost sure.
Then the event
\begin{equation} \label{def.Omega2}
\Omega_2=\Omega_1\cap\Omega_X
\end{equation}
is almost sure, and there exists a positive and deterministic $\underline{X}$ given by
\begin{equation} \label{def.Xunderline}
\underline{X}=\inf_{\omega\in \Omega_2} \limsup_{t\to\infty} |X(t,\omega)|.
\end{equation}
\item[(C)] Suppose that there is an a.s. event  defined by
\[
\{\omega\in \Omega_Y:\limsup_{t\to \infty}|Y(t, \omega)|=+\infty\}.
\]
Then $\limsup_{t\to \infty}|X(t)|=+\infty$ a.s.
\end{itemize}
\end{proposition}

In the proof of part (B), we can even determine an explicit lower bound for $\underline{X}$.
If the event $\Omega_1$ is defined by \eqref{def.Omega1}, we may define as in \eqref{def.underY} and
\eqref{def.overY} the deterministic numbers $0<\underline{Y}\leq \overline{Y}<+\infty$. For any $f$ obeying \eqref{eq.fglobalunperturbed} 
it can be shown that there is function $y\mapsto\underline{x}(y)=\underline{x}(f,y)$
which, for $y\geq 0$, obeys
\begin{equation} \label{def.lambda}
2\underline{x} + \max_{|x|\leq \underline{x}} |f(x)|=y.
\end{equation}
This leads to the estimate
\begin{equation} \label{eq.Xunderlinelower}
\underline{X}\geq \underline{x}(f,\underline{Y}),
\end{equation}
where $\underline{Y}$ is given by \eqref{def.underY}. Moreover, as it transpires that $\underline{x}(f,\cdot)$ is an increasing function, by \eqref{eq.estYbarYunderbar}, we can estimate $\underline{X}$ explicitly according to
\[
\underline{X}\geq \underline{x}(f,\underline{y}),
\]
where $\underline{y}$ is given explicitly by \eqref{eq.estYbarYunderbar}.

An interesting implication of part (C) is that an \emph{arbitrarily
strong} mean--reverting force (as measured by $f$) cannot keep
solutions of \eqref{eq.sde} within bounded limits if the noise
perturbation is so intense that  a linear mean--reverting force
cannot keep solutions bounded. Therefore, the system will run
``out of control'' (in the sense of becoming unbounded) however
strongly the function $f$ pushes it back towards the equilibrium
state.
\subsection{Main results}
Due to Theorem~\ref{theorem.Ybounded}, we can readily use
Proposition~\ref{prop.yimpliesx} to characterise the asymptotic
behaviour of solutions of \eqref{eq.sde}.
\begin{theorem}  \label{theorem.XtozeroXunbdd}
Suppose that $\sigma$ obeys \eqref{eq.sigcns}, $f$ obeys
\eqref{eq.fglobalunperturbed} and that $X$ is any continuous
adapted process which obeys \eqref{eq.sde}. Let $\theta$ be defined
by \eqref{def.theta} and $S'(\cdot)$ by \eqref{def.s'epsilon}.
\begin{itemize}
\item[(A)] If $\theta$ is such that \eqref{eq.thetastable} holds, then
\[
\lim_{t\to\infty} X(t)=0, \quad\text{a.s.}
\]
\item[(B)] If $\theta$ is such that \eqref{eq.thetabounded} holds, then there exists an almost sure event
$\Omega_2=\Omega_1\cap \Omega_X$, and a deterministic $\underline{X}>0$
defined by \eqref{def.Xunderline} such that
\[
\underline{X}=\inf_{\omega\in \Omega_2} \limsup_{t\to\infty} |X(t,\omega)|>0.
\]
Moreover, $\underline{X}$  obeys
\[
\underline{X}\geq \underline{x}(f,\underline{Y}),
\]
where $\underline{x}(f,\cdot)$ is the unique solution of \eqref{def.lambda}, and $\underline{Y}$ is defined by \eqref{def.underY}. Furthermore,
\[
\liminf_{t\to\infty} |X(t)|=0, \quad \text{a.s.}
\]
\item[(C)] If $\theta$ is such that \eqref{eq.thetaunstable} holds, then
\[
\limsup_{t\to\infty} |X(t)|=+\infty, \quad \liminf_{t\to\infty}
|X(t)|=0, \quad \text{a.s.}
\]
\end{itemize}
\end{theorem}
\begin{proof}
If $\theta$ is such that \eqref{eq.thetastable} holds, then from
Theorem~\ref{theorem.Ybounded}, we have $\lim_{t\to \infty}Y(t)=0$,
a.s. Taking this together with Proposition~\ref{prop.yimpliesx},
part (A) holds. If $\theta$ is such that \eqref{eq.thetabounded}
holds, or if $\theta$ is such that \eqref{eq.thetaunstable} holds,
then taken together with Theorem~\ref{theorem.Ybounded} and
Proposition~\ref{prop.yimpliesx} we have that the first part (B) and
of (C) is true. For the second part of (B) and (C), we recall that
if \eqref{eq.thetabounded} or \eqref{eq.thetaunstable} hold,
Remark~\ref{rem.sigL2} implies that $\sigma \notin L^{2}(0,\infty).$
In this case, we already know that $\liminf_{t\to \infty}|X(t)|=0$,
a.s. by Theorem~\ref{lemma.Xliminf}.
\end{proof}

The formula \eqref{def.lambda}, which is established in the proof of
part (B) of Proposition~\ref{prop.yimpliesx}, relates the lower
bound on the large fluctuations $\underline{x}$  to the size of the
diffusion coefficient $\sigma$ and the nonlinearity in $f$. Thus, we
may view
$\underline{x}=\underline{x}(f,\underline{Y})=\underline{x}(f,\sigma)$,
because $\underline{Y}$ depends on $\sigma$ but not on $f$. It is
clear that the larger the diffusion coefficient, the larger the
value of $\underline{Y}$. We now show for fixed $f$ that
$\underline{x}$ is increasing  and that
$\underline{x}(f,y)\to \infty$ as
$y\to\infty$. Moreover, we show for fixed
$y$ that $\underline{x}(f_1,y)\geq
\underline{x}(f_2,y)$ if
\begin{equation} \label{eq.f2f1}
|f_2(x)|\geq |f_1(x)|, \quad x\in \mathbb{R}.
\end{equation}
These ordering results seem to make intuitive sense, as we would
expect weaker mean reversion and a larger diffusion coefficient to
lead to larger fluctuations in $X$.
\begin{proposition}  \label{prop.underline}
Suppose that 
$f$ obeys
\eqref{eq.fglobalunperturbed}. 
Let $\underline{x}$ be the
unique solution of \eqref{def.lambda}.
Then
\begin{itemize}
\item[(i)] $y\mapsto \underline{x}(f,y)$ is increasing and   $\lim_{y\to\infty}\underline{x}(f,y)=+\infty$, $\lim_{y\to 0^+}\underline{x}(f,y)=0$.
\item[(ii)]
If $f_1$ and $f_2$ are functions that obey \eqref{eq.fglobalunperturbed}
and also satisfy \eqref{eq.f2f1}, then
$\underline{x}(f_1,y)\geq
\underline{x}(f_2,y)$.
\end{itemize}
\end{proposition}
\begin{proof}
Define $h_f:[0,\infty)\to [0,\infty)$ according to
\begin{equation} \label{def.hf}
h_f(x):=2x + \max_{|y|\leq x}|f(y)|, \quad x\geq 0.
\end{equation}
Then $h_f$ is increasing and continuous, and obeys the limits
$\lim_{x\to\infty}h_f(x)=+\infty$ and $\lim_{x\to 0^+}h_f(x)=0$. By
\eqref{def.lambda},
$h_f(\underline{x}(f,y))=y$. Therefore
\begin{equation} \label{eq.underxhfinv}
\underline{x}(f,y)=h_f^{-1}(y), \quad y \geq 0.
\end{equation}
Hence $y\mapsto \underline{x}(f,y)$ is increasing.
Finally, $\lim_{y\to\infty}
\underline{x}(f,y)=\infty$ and $\lim_{y\to
0^+} \underline{x}(f,y)=\lim_{y\to 0^+}
h_f^{-1}(y)$=0.

To prove part (ii), note by \eqref{eq.f2f1} that
\begin{align*}
h_{f_1}(\underline{x}(f_1,y))&=y=h_{f_2}(\underline{x}(f_2,y))
=2\underline{x}(f_2,y)+\max_{|u|\leq \underline{x}(f_2,y)} |f_2(u)|\\
&\geq 2\underline{x}(f_2,y)+\max_{|u|\leq \underline{x}(f_2,y)}
|f_1(u)| =h_{f_1}(\underline{x}(f_2,y)).
\end{align*}
Since $h_{f_1}$ is an increasing function, we have
$\underline{x}(f_1,y)\geq \underline{x}(f_2,y)$ as required.
\end{proof}
Just as the conditions of Theorem~\ref{theorem.Ybounded} can be
quite difficult to check in practice for $Y$, the same is true for
the conditions of Theorem~\ref{theorem.XtozeroXunbdd} on $\theta$
for $X$. As in Theorem~\ref{theorem.sufficient}, and because of
Proposition~\ref{prop.yimpliesx}, we can supply easily checked
sufficient conditions on $\sigma$ for which $X$ is bounded, stable
or unstable.

In the case when $\sigma\in L^2(0,\infty)$ we have that $X$ tends to
zero. Therefore, we confine attention to the case where
$\sigma\not\in L^2(0,\infty)$. In this case, we can define a number
$T>0$ such that $\int_0^t e^{2s}\sigma^2(s)\,ds>e^e$ for $t>T$ and
so one can define, as before, the function $\Sigma:[T,\infty)\to
[0,\infty)$ by \eqref{def.Sigma}.
\begin{theorem} \label{theorem.sufficientforx}
Suppose that $f$ obeys \eqref{eq.fglobalunperturbed} and that $\sigma$ obeys \eqref{eq.sigcns}.
Let $X$ be any solution of \eqref{eq.sde}. Let $\Sigma$ be given by \eqref{def.Sigma}.
\begin{itemize}
\item[(A)] If $\lim_{t\to \infty}\Sigma^2(t)=0$
then $\lim_{t\to\infty} X(t)=0$ a.s.
\item[(B)]
If there exists $L\in (0,\infty)$ such that $\liminf_{t\to \infty}
\Sigma^2(t)=L$, then there exists an almost sure event
$\Omega_2=\Omega_1\cap \Omega_X$, and a deterministic $\underline{X}>0$
defined by \eqref{def.Xunderline} such that
$\underline{X}=\inf_{\omega\in \Omega_2} \limsup_{t\to\infty} |X(t,\omega)|>0$.
Moreover, $\underline{X}\geq \underline{x}(f,\underline{Y})$, where $\underline{x}(f,\cdot)$ is the unique solution of \eqref{def.lambda}, and $\underline{Y}$ is defined by \eqref{def.underY}.
\item[(C)] If $\lim_{t\to \infty} \Sigma^2(t)=+\infty$
then $\limsup_{t\to\infty} |X(t)|=+\infty$ a.s.
\end{itemize}
\end{theorem}
\begin{proof}
If $\lim_{t\to \infty}e^{-2t}\log
t\int_{0}^{t}e^{2s}\sigma^{2}(s)ds=0$ then $\lim_{t\to\infty}Y(t)=0$
from Theorem~\ref{theorem.sufficient}. Combining this with
Proposition~\ref{prop.yimpliesx}, we get $\lim_{t\to \infty}X(t)=0$
proving part (A). Similarly, parts (B) and (C) follow from parts (B)
and (E) of Theorem~\ref{theorem.sufficient} and
Proposition~\ref{prop.yimpliesx}.
\end{proof}
We finish this Section by giving a sufficient condition on $f$ for
which solutions of \eqref{eq.sde} do not tend to zero but are
nonetheless bounded. In the case when $\sigma$ is such that either
parts (A) or (C) apply, we have unambiguous information about the
asymptotic behaviour of solutions: either almost all sample paths
tend to zero, or almost all sample paths exhibit unbounded
fluctuations. However, scrutiny of the statement of
Proposition~\ref{prop.yimpliesx} shows that part (B) does not rule
out the possibility that $\limsup_{t\to\infty} |X(t)|=+\infty$ with
positive probability (or even almost surely). We make a further
hypothesis on $f$, under which  this is impossible, and $X$ is
forced to be bounded. The hypothesis is
\begin{equation} \label{h.finfty}
\lim_{x\to-\infty} f(x)=-\infty, \quad \lim_{x\to\infty}
f(x)=\infty.
\end{equation}

An estimate on the lower bound $\underline{X}$ in case (B) is given
in \eqref{def.lambda}, which is found as part of the proof of
Proposition~\ref{prop.yimpliesx}. $\underline{X}$ is given in terms
of $f$ and $\sigma$. Similarly, an estimate can be determined for
the upper bound. Towards this end, we introduce functions which are a type of generalised inverse of $f$ by
defining the functions $f^-$ and
$f^+$ by
\begin{align} \label{def.fgeninv}
f^+(x)&=\sup\{z>0:f(z)=x\}, \quad x\geq 0,\\
 \label{def.fgeninvneg}
f^-(x)&=\inf\{z<0:f(z)=x\}, \quad x\leq 0.
\end{align}
These functions are well--defined if $f$ obeys \eqref{eq.fglobalunperturbed} 
and \eqref{h.finfty}. We notice also that if $f$ is increasing, then $f^{\pm}$ are exactly the inverse of $f$.

We may therefore define for any $f$ the function $y\mapsto \overline{x}(f,y)$ by
\begin{equation}\label{def.overlineX}
\overline{x}(f,y)=2y+\max(f^+(y),-f^-(-y)), \quad y\geq 0.
\end{equation}
The main conclusion of the following theorem is that an explicit upper bound can be found for $\limsup_{t\to\infty} |X(t)|$. In fact, it can be shown that if $\overline{Y}$ obeys \eqref{def.overY}, then
\begin{equation} \label{eq.limsupXuprboundfinal}
\limsup_{t\to\infty} |X(t,\omega)|\leq \overline{x}(f,\overline{Y}), \quad\text{for each $\omega\in \Omega_2$},
\end{equation}
where $\Omega_2$ is given by \eqref{def.Omega2}.

We are finally in a position to state the main result of this section.
\begin{theorem}  \label{theorem.Xbounded}
Suppose that $\sigma$ obeys \eqref{eq.sigcns}, $f$ obeys
\eqref{eq.fglobalunperturbed} and \eqref{h.finfty}.
Suppose that $X$ is any continuous adapted process which
obeys \eqref{eq.sde}. Let $\theta$ be defined by \eqref{def.theta}
and $S'(\cdot)$ by \eqref{def.s'epsilon}.
\begin{itemize}
\item[(A)] If $\theta$ is such that \eqref{eq.thetastable} holds, then $\lim_{t\to\infty} X(t)=0$, a.s.
\item[(B)] If $\theta$ is such that \eqref{eq.thetabounded} holds, then there exists an almost sure event
$\Omega_2=\Omega_1\cap \Omega_X$ where $\Omega_1$ defined in \eqref{def.Omega1}, and deterministic
$0<\underline{X}\leq \overline{X}<+\infty$ such that
\begin{equation} \label{eq.overlineunderlineX}
\underline{X}=\inf_{\omega\in \Omega_2} \limsup_{t\to\infty} |X(t,\omega)|, \quad
\overline{X}=\sup_{\omega\in \Omega_2} \limsup_{t\to\infty} |X(t,\omega)|,
\end{equation}
Moreover,
\[
\underline{X}\geq \underline{x}(f,\underline{Y}),
\]
where $\underline{x}(f,\cdot)$ is the unique solution of \eqref{def.lambda}, and $\underline{Y}$ is defined by \eqref{def.underY},
and
\[
\overline{X}\leq \overline{x}(f,\overline{Y}),
\]
where $\overline{x}(f,\cdot)$ is defined by \eqref{def.overlineX} and $\overline{Y}$ is defined by \eqref{def.overY}. Furthermore,
\begin{equation*}  
\liminf_{t\to\infty}|X(t)|=0, \quad \text{a.s.}
\end{equation*}
\item[(C)] If $\theta$ is such that \eqref{eq.thetaunstable} holds, then
\[
\limsup_{t\to\infty} |X(t)|=+\infty, \quad \liminf_{t\to\infty}
|X(t)|=0, \quad \text{a.s.}
\]
\end{itemize}
\end{theorem}
We prove part (B) only, as the results of parts (A) and (C) follow
from Theorem~\ref{theorem.XtozeroXunbdd}. Therefore, under the
additional hypothesis that $f$ obeys \eqref{h.finfty}, it follows
from Theorem~\ref{theorem.XtozeroXunbdd} and \ref{theorem.Xbounded}
that either (i) solutions tend to zero with probability one, when
$\sigma$ obeys \eqref{eq.thetastable} (ii) solutions fluctuate
within finite bounds with probability one, when $\sigma$ obeys
\eqref{eq.thetabounded} or (iii) solutions fluctuate unboundedly
with probability one, when $\sigma$ obeys \eqref{eq.thetaunstable}. In the second case, part (B) of Theorem~\ref{theorem.Xbounded} can be restated as
\[
\underline{x}(f,\underline{Y})\leq \limsup_{t\to\infty} |X(t)|\leq  \overline{x}(f,\overline{Y}), \quad \text{a.s.},
\]
and moreover we have weaker but explicit estimates on these deterministic bounds given by
\[
0<\underline{x}(f,\underline{y})\leq \limsup_{t\to\infty} |X(t)|\leq  \overline{x}(f,\overline{y})<+\infty, \quad \text{a.s.},
\]
where $\underline{y}$ and $\overline{y}$ are given by \eqref{eq.estYbarYunderbar}.


It is interesting to determine the effect of weaker mean reversion and an increasing diffusion coefficient on the \emph{upper} bound of the large deviations of $X$, given by $\overline{x}(f,\overline{Y})$, just as we did for the lower bound on the size of the largest fluctuations in Proposition~\ref{prop.underline}, given
by $\underline{x}(f,\underline{Y})$. As before, it can be shown that weaker mean reversion and increasing diffusion
coefficients increase the bound $\overline{x}$. Also, if the effect of the diffusion coefficient alone is negligible (so that $\overline{Y}\to 0$), or unboundedly large (so that $\overline{Y}\to \infty$), we see that cases (A) and (C) in Theorem~\ref{theorem.Xbounded} can be viewed as limiting cases of the asymptotic behaviour
described in case (B). These properties of the bounds are established in the following result.
\begin{proposition} \label{prop.overline}
Suppose that 
$f$ obeys \eqref{eq.fglobalunperturbed}
and \eqref{h.finfty}.
Let
$\overline{x}$ be given by \eqref{def.overlineX}. Then
\begin{itemize}
\item[(i)] $y\mapsto \overline{x}(f,y)$ is increasing and $\lim_{y\to\infty}\overline{x}(f,y)=+\infty$, $\lim_{y\to 0^+}\overline{x}(f,y)=0$.
\item[(ii)] If $f_1$ and $f_2$ are functions that obey \eqref{eq.fglobalunperturbed} and
\eqref{h.finfty}, and also satisfy
\eqref{eq.f2f1}, then $\overline{x}(f_1,y)\geq
\overline{x}(f_2,y)$.
\end{itemize}
\end{proposition}
The proof is relegated to the final section. We finish the section with an example which shows how estimates of $\underline{X}$ and $\overline{X}$
can be obtained in practice.
\begin{example} \label{examp.1}
We see how these estimates on the fluctuations behave for a specific
class of examples. Suppose that $f(x)=x^n$ where $n$ is an odd
integer and that $\sigma^2(t)\log t\to L\in (0,\infty)$ as
$t\to\infty$. Then by Theorem~\ref{theorem.sufficient} it follows
that $\limsup_{t\to\infty} |Y(t)|=\sqrt{2L}$ a.s. so we have
$\overline{Y}=\underline{Y}=\sqrt{2L}$. Since $f$ is increasing we
have for $x\geq 0$ that
\[
f^+(x)=f^{-1}(x)=x^{1/n}, \quad f^{-}(-x)=f^{-1}(-x)=-x^{1/n}, \quad
\max_{|y|\leq x} |f(y)|=x^n
\]
so that $\underline{x}(L)=\underline{x}(f,\underline{Y})$ and $\overline{x}(L)=\overline{x}(f,\overline{Y})$ satisfy
\[
2\underline{x}+\underline{x}^{n}=\sqrt{2L}, \quad
\overline{x}=2\sqrt{2L}+(\sqrt{2L})^{1/n}.
\]
From this, we readily see that
\[
\lim_{L\to 0^+} \frac{\underline{x}(L)}{\sqrt{2L}}=\frac{1}{2}, \quad
\lim_{L\to 0^+} \frac{\overline{x}(L)}{(\sqrt{2L})^{1/n}}=1,
\]
and that
\[
\lim_{L\to \infty} \frac{\underline{x}(L)}{(\sqrt{2L})^{1/n}}=1, \quad
\lim_{L\to \infty} \frac{\overline{x}(L)}{2(\sqrt{2L})}=1.
\]
Notice that $\lim_{L\to 0^+} \overline{x}(L)=0$ and $\lim_{L\to \infty} \underline{x}(L)=\infty$.

It is clear that these asymptotic bounds are widely spaced, because
\[
\lim_{L\to 0^+} \frac{\overline{x}(L)}{\underline{x}(L)}=\lim_{L\to \infty} \frac{\overline{x}(L)}{\underline{x}(L)}=+\infty.
\]
It would be an interesting question to determine whether either of these bounds is satisfactory, but we do not pursue this here. We suspect that the upper bound $\overline{x}(L)$ as $L\to\infty$ is very conservative, however, as it does not take into account the strong mean reversion of $f$.
\end{example}

\section{Asymptotic Stability}
It should be remarked that one consequence of
Theorem~\ref{theorem.XtozeroXunbdd} is that sample paths of $X$ tend
to zero with non--zero probability if and only if $\theta$ obeys
\eqref{eq.thetastable}, in which case almost all sample paths tend
to zero. Therefore, we have the following immediate corollary of
Theorem~\ref{theorem.XtozeroXunbdd}.
\begin{theorem} \label{theorem.Xiffsigma}
Suppose $f$ obeys \eqref{eq.fglobalunperturbed} and
that $\sigma$ obeys \eqref{eq.sigcns}.
 Let $X$ be any solution of \eqref{eq.sde}. Let $\theta$ be defined by
\eqref{def.theta} and let $\Phi$ be given by \eqref{def.phi}. Then
the following are equivalent:
\begin{itemize}
\item[(A)]
\begin{equation}  \label{eq.BC1}
\sum_{n=1}^\infty \theta(n)
\exp\left(-\frac{1}{2}\frac{\epsilon^2}{\theta^2(n)}\right) <
+\infty, \quad \text{for every $\epsilon>0$}.
\end{equation}
\item[(B)] $\lim_{t\to\infty} X(t,\xi)=0$ with positive probability for some $\xi\in\mathbb{R}$.
\item[(C)] $\lim_{t\to\infty} X(t,\xi)=0$ a.s. for each $\xi\in\mathbb{R}$.
\end{itemize}
\end{theorem}
Part (A) refines part of~\cite[Proposition 3.3]{JAJGAR:2009}. Also,
if $X(t)\to 0$ as $t\to\infty$, it does so a.s., and so $\theta$
obeys \eqref{eq.thetastable}. Therefore, $Y(t)\to 0$ as
$t\to\infty$. This forces $\liminf_{t\to\infty} \Sigma^2(t)=0$, for
else we would have $\limsup_{t\to\infty} |Y(t)|>0$ a.s, as
essentially pointed out by~\cite[Proposition 3.3]{JAJGAR:2009}.

It should also be noted that no monotonicity conditions are required
on $\sigma$ in order for this result to hold, and that a.s. global
stability is independent of the form of $f$. The conditions and form
of Theorem~\ref{theorem.XtozeroXunbdd} and \ref{theorem.Xiffsigma}
are inspired by those of \cite[Theorem 1]{ChanWill:1989} and by
\cite[Theorem 6, Corollary 7]{JAARMR:2009}.

An interesting fact of Theorem~\ref{theorem.Xiffsigma} is that it
is unnecessary for $\sigma(t)\to 0$ as $t\to\infty$ in order for $X$
to obey \eqref{eq.stochglobalstable}. In fact, we can even have
$\limsup_{t\to\infty} |\sigma(t)|^2=\infty$ and still have $X(t)\to
0$ as $t\to\infty$ a.s. Some examples are supplied in~\cite{JAJGAR:2009}.

Note that \eqref{eq.sigmalogto0} implies
$\lim_{t\to\infty}\Sigma(t)=0$, that \eqref{eq.sigmalogtoinfty}
implies $\lim_{t\to\infty}\Sigma(t)=\infty$, and finally that
$\liminf_{t\to\infty} \sigma^2(t)\log t>0$ implies that
$\liminf_{t\to\infty} \Sigma(t)>0$. The next result is therefore an
easy corollary of Theorem~\ref{theorem.sufficient}, or of
Proposition~\ref{prop.yimpliesx} and
Proposition~\ref{prop.sigmasqlogtS}.
\begin{theorem} \label{theorem.stochastablegen}
Suppose that $f$ satisfies \eqref{eq.fglobalunperturbed},
and that $\sigma$ obeys \eqref{eq.sigcns}.
Let $X$ be any solution of \eqref{eq.sde}.
\begin{itemize}
\item[(i)] If $\sigma$ obeys $\lim_{t\to\infty} \sigma^2(t)\log t=0$, then $X$ obeys \eqref{eq.stochglobalstable}.
\item[(ii)] If $\sigma$ obeys $\liminf_{t\to\infty} \sigma^2(t)\log t\in (0,\infty)$, then
$\mathbb{P}[\lim_{t\to\infty} X(t)=0]=0$ and $\liminf_{t\to\infty}
|X(t)|=0$, a.s..
\item[(iii)] If $\sigma$ obeys $\lim_{t\to\infty} \sigma^2(t)\log t =\infty$, then
\[
\limsup_{t\to\infty} |X(t)|=\infty, \quad \liminf_{t\to\infty}
|X(t)|=0, \quad\text{a.s.}
\]
\end{itemize}
\end{theorem}
Part (i) is part of~\cite[Proposition 3.3(a)]{JAJGAR:2009}. Part
(iii) is~\cite[Lemma 3.7]{JAJGAR:2009}. In \cite{ChanWill:1989},
Chan and Williams have proven in the case when $t\mapsto\sigma^2(t)$
is decreasing, that $Y$ obeys \eqref{eq.Y0} if and only if $\sigma$
obeys \eqref{eq.sigmalogto0}. Our final result is a corollary of
this observation and Theorem~\ref{theorem.Xiffsigma}, and also of
\cite[Theorem 3.8]{JAJGAR:2009}. A stronger result than
Theorem~\ref{theorem.stochastablegen} can be stated if the sequence
$\theta$ in \eqref{def.theta} is decreasing: in this case,
$\lim_{n\to\infty}\theta^2(n)\log n=0$ is equivalent to
\eqref{eq.stochglobalstable}.
\begin{theorem} \label{theorem.eqvt}
Suppose that $f$ satisfies \eqref{eq.fglobalunperturbed}.
Suppose that $\sigma$ obeys \eqref{eq.sigcns} and $t\mapsto\sigma^2(t)$ is decreasing. Let $X$ be any solution of
\eqref{eq.sde}. Then the following are equivalent:
\begin{itemize}
\item[(A)] $\sigma$ obeys $\lim_{t\to\infty} \sigma^2(t)\log t=0$;
\item[(B)] $\lim_{t\to\infty} X(t,\xi)=0$ a.s. for each $\xi\in\mathbb{R}$.
\end{itemize}
\end{theorem}
The remark preceding this result points to the importance of the
condition $\theta(n)^2\log n\to 0$ as $n\to\infty$, as indeed does
Proposition~\ref{prop.thetaSigma} part (a). We now supply an example
in which $\theta(n)^2\log n\to 0$ as $n\to\infty$, but $t\mapsto
\sigma^2(t)$ has ``spikes'' which prevents it from satisfying the
condition $\lim_{t\to\infty} \sigma^2(t)\log t=0$.
\begin{example}
Consider the decomposition of $[0,\infty)$ into a union of disjoint intervals
\[
[0,\infty)=\cup_{k=0}^\infty \{I_k \cup J_k \cup  K_k\},
\]
where $\epsilon_k\in (0,1/2)$ for each $k\geq 0$ and
\[
I_k=[k, k+\epsilon_k],\quad J_k=(k+\epsilon_k, k+1-\epsilon_k),\quad K_k=[k+1-\epsilon_k, k+1), \quad k\in \mathbb{N}.
\]
Let $(l_k)_{k\geq 0}$ and $(q_k)_{k\geq 0}$ be positive sequences and consider the function $\sigma:[0,\infty)\to [0,\infty)$ defined by
\[
\sigma^2(t)= \left\{
\begin{array}{ccc}
& l_k-\frac{l_k-q_k}{\epsilon_k}(t-k),&  \quad t\in [k, k+\epsilon_k] ,   \\
& q_k, &  \quad t\in (k+\epsilon_k, k+1-\epsilon_k),   \\ & l_{k+1}+\frac{l_{k+1}-q_k}{\epsilon_k}(t-k-1),
&  \quad t\in [k+1-\epsilon_k, k+1).\end{array}\right.
\]
Then $t\mapsto \sigma^2(t)$ is continuous. If $\theta$ is defined by
\eqref{def.theta}, then
\begin{align*}
\theta^2(k)
&=q_k(1-\epsilon_k)+\frac{1}{2}\epsilon_k(l_{k+1}+l_k).
\end{align*}
Notice also that $\sigma^2(k)=l_k$. Suppose $q_k \log k\to 0$,
$\epsilon_k(l_k+l_{k+1})\log k\to 0$ but
$\limsup_{k\to\infty}l_k\log k>0$. Then $\theta^2(k)\log k\to 0$ as
$k\to\infty$, but
\[
\limsup_{t\to\infty} \sigma^2(t)\log t\geq \limsup_{k\to\infty}
\sigma^2(k)\log k= \limsup_{k\to\infty} l_k\log k>0.
\]
Concrete examples of sequences for which these conditions hold
include
\[
q_k=\frac{1}{k+1}, \quad \epsilon_k=\frac{1}{k+3}, \quad l_k=1,
\]
or
\[
q_k=\frac{1}{k+1}, \quad  \epsilon_k=\frac{1}{(k+3)^2}, \quad l_k=k.
\]
\end{example}

\section{Proof of Existence Results from Section~\ref{sec:exist}} \label{sec.proofsexist}
\subsection{Proof of Proposition~\ref{prop.exist}}
Consider the affine stochastic differential equation
\eqref{eq.linsde}. Since $\sigma$ is continuous, there is a unique
continuous adapted process which obeys \eqref{eq.linsde}. Let
$\Omega_Y$ be the a.s. event defined by \eqref{def.OmegaY} on which
$Y$ is defined. Now, for each $\omega\in \Omega_Y$, define the
function
\[
\varphi(t,x,\omega)=-f(x+Y(t,\omega))+Y(t,\omega), \quad t\geq 0.
\]
Since $f$ is continuous, and the sample path $t\mapsto Y(t,\omega)$
is continuous, $(t,x)\mapsto \varphi(t,x,\omega)$ is continuous.
Consider now the differential equation
\[
z'(t,\omega)=\varphi(t,z(t,\omega),\omega), \quad t>0; \quad
z(0,\omega)=\xi.
\]
By the continuity of $\varphi$ in both arguments, by the Peano
existence theorem, there exists a continuous local solution
$t\mapsto z(t,\omega)$ for each $\omega\in \Omega_Y$ and $0\leq
t<\tau_e(\omega)$.
Presently, it will be shown that $\tau_e(\omega)=+\infty$ a.s. on
$\Omega_Y$.

Moreover, as $Y$ is adapted to $(\mathcal{F}^B(t))_{t\geq 0}$, $z$
is also adapted to $(\mathcal{F}^B(t))_{t\geq 0}$. Now consider the
process $X$ defined on $\Omega_Y$ by $X(t)=z(t)+Y(t)$ for $t\in
[0,\tau_e)$. By construction it is continuous and adapted.
Furthermore, we have for $t\in [0,\tau_e)$
\begin{align*}
X(t,\omega)&=z(t,\omega)+Y(t,\omega)\\
&=\xi+\int_0^t \varphi(s,z(s,\omega),\omega)\,ds + \int_{0}^t -Y(s,\omega)\,ds + \left(\int_0^t \sigma(s)\,dB(s)\right)(\omega)\\
&=\xi+\int_0^t \left\{-f(z(s,\omega)+Y(s,\omega))+Y(s,\omega)\right\}\,ds + \int_{0}^t -Y(s,\omega)\,ds \\
&\qquad + \left(\int_0^t \sigma(s)\,dB(s)\right)(\omega)\\
&=\xi+\int_0^t -f(X(s,\omega))\,ds + \left(\int_0^t \sigma(s)\,dB(s)\right)(\omega)\\
&=\left(\xi+\int_0^t -f(X(s))\,ds + \int_0^t
\sigma(s)\,dB(s)\right)(\omega).
\end{align*}
Hence $X(\cdot,\omega)$ obeys \eqref{eq.sde} for each $\omega\in
\Omega_Y$ on the interval $[0,\infty)$. The proof that $\tau_e$ is
infinite a.s. was given in the Appendix of \cite{JAJGAR:2009}.

\subsection{Proof of Proposition~\ref{prop.existonesided}}
The proof is inspired by an observation in e.g., \cite{K&S}. Note
first by Proposition~\ref{prop.exist} that the continuity of $f$
together with \eqref{eq.fglobalunperturbed} guarantees the existence
of a continuous adapted process which obeys \eqref{eq.sde}. Suppose
therefore that $X_1$ and $X_2$ are any two solutions of
\eqref{eq.sde}. Then
\[
d(X_1(t)-X_2(t))=\left(-f(X_1(t))+f(X_2(t))\right)\,dt,
\]
and by It\^o's rule we have that
\[
d(X_1(t)-X_2(t))^2=-2(X_1(t)-X_2(t))\left(f(X_1(t))-f(X_2(t))\right)\,dt,
\quad t\geq 0.
\]
Since $X_1(0)=X_2(0)=\xi$, we have
\[
(X_1(t)-X_2(t))^2=-2\int_0^t
(X_1(s)-X_2(s))\left(f(X_1(s))-f(X_2(s))\right)\,ds, \quad t\geq 0.
\]
Since $f$ obeys \eqref{eq.fonesidedlip}, we have
\[
(X_1(t)-X_2(t))^2\leq 2K\int_0^t (X_1(s)-X_2(s))^2\,ds, \quad t \geq
0.
\]
If $K\leq 0$, we can conclude automatically that $X_1(t)=X_2(t)$ for
all $t\geq 0$ a.s., and that therefore the solution is unique. If
$K>0$, by applying Gronwall's inequality to the non--negative
continuous function $t\mapsto (X_1(t)-X_2(t))^2$, we conclude that
$X_1(t)=X_2(t)$ for all $t\geq 0$ a.s., and once again we have
uniqueness.

\section{Proofs of Preliminary Results}  \label{sec.prelimproof}
\subsection{Proof of Theorem~\ref{th.sigL2}} \label{sec.sigmal2proof}
By It\^o's rule, we have
\begin{equation} \label{eq.Xsq}
X^2(t)=\xi^2 - \int_0^t 2X(s)f(X(s))\,ds+\int_0^t \sigma^2(s)\,ds +
\int_0^t 2X(s)\sigma(s)\,dB(s), \quad t\geq 0.
\end{equation}
Since $xf(x)\geq 0$ for all $x\in\mathbb{R}$ and $\sigma\in
L^2(0,\infty)$, we have
\[
X^2(t)\leq \xi^2+\int_0^\infty \sigma^2(s)\,ds + 2\int_0^t
X(s)\sigma(s)\,dB(s), \quad t\geq 0.
\]
Define $M$ to be the local martingale given by $M(t)=\int_0^t
X(s)\sigma(s)\,dB(s)$ for $t\geq 0$. Suppose that there is an event
$A=\{\omega: \lim_{t\to\infty} \langle M \rangle (t)=+\infty\}$ such
that $\mathbb{P}[A]>0$. Then a.s. on $A$ we have
$\liminf_{t\to\infty} M(t)=-\infty$, which implies that
$\liminf_{t\to\infty} X^2(t)=-\infty$ a.s. on $A$, which is absurd.
Therefore $\lim_{t\to\infty} \langle M\rangle(t)<+\infty$ a.s., so
it follows that $\lim_{t\to\infty} M(t)=:M(\infty)$ exists a.s. and
is a.s. finite. Therefore we have that $t\mapsto |X(t)|$ is a.s.
bounded. Since $X^2(t)\geq 0$, it follows from \eqref{eq.Xsq} that
\[
\int_0^t 2X(s)f(X(s))\,ds =\xi^2 -X^2(t)+\int_0^t \sigma^2(s)\,ds +
2M(t) \leq \xi^2 +\int_0^\infty \sigma^2(s)\,ds + 2M(t).
\]
Therefore, as $xf(x)\geq 0$ for all $x\in\mathbb{R}$ and $M(t)\to
M(\infty)$ as $t\to\infty$ (where $M(\infty)$ is finite), we have
that
\[
\lim_{t\to\infty} \int_0^t X(s,\omega)f(X(s,\omega))\,ds
=:I(\omega)\in (0,\infty)
\]
Therefore, as $t\mapsto |X(t)|$ is a.s. bounded, and all the terms
on the righthand side of \eqref{eq.Xsq} have finite limits as
$t\to\infty$, it follows that there is $L=L(\omega)\in [0,\infty)$
such that $\lim_{t\to\infty} X^2(t,\omega)=L(\omega)$ for all
$\omega$ in an a.s. event $A$, say. By continuity this means that
there is an a.s. event $A=\{\omega: X(t,\omega)^2\to L\in [0,\infty)
\text{ as $t\to\infty$}\}$ such that $A=A_+\cup A_-\cup A_0$ where
\begin{gather*}
A_+=\{\omega: X(t,\omega)\to \sqrt{L(\omega)}\in (0,\infty)\text{ as $t\to\infty$}\}, \\
 A_-=\{\omega: X(t,\omega)\to -\sqrt{L(\omega)}\in (-\infty,0)\text{ as $t\to\infty$}\},
\end{gather*}
and $A_0=\{\omega: X(t,\omega)\to 0 \text{ as $t\to\infty$}\}$.
Suppose that $\omega\in A_+$. Then
\begin{equation} \label{eq.intxfxave}
\lim_{t\to\infty} \frac{1}{t}\int_0^t X(s,\omega)f(X(s,\omega))\,ds
= \sqrt{L(\omega)}f(\sqrt{L(\omega)})>0,
\end{equation}
by continuity of $X$, $f$ and the fact that $xf(x)>0$ for $x\neq 0$.
Since the last two terms on the righthand side of \eqref{eq.Xsq}
have finite limits as $t\to\infty$, \eqref{eq.intxfxave} implies
that for $\omega\in A_+$ that
\[
0\leq \lim_{t\to\infty}
\frac{X^2(t,\omega)}{t}=-2\sqrt{L(\omega)}f(\sqrt{L(\omega)})<0,
\]
a contradiction. Therefore $\mathbb{P}[A_+]=0$. A similar argument
yields $\mathbb{P}[A_-]=0$. Since $\mathbb{P}[A]=1$, we must have
$\mathbb{P}[A_0]=1$, as required.

\subsection{Proof of Theorem~\ref{lemma.Xliminf}}
Let $A_1=\{\omega: \liminf_{t\to\infty} X(t,\omega)>0\}$ and
suppose that $\mathbb{P}[A_1]>0$. In particular, for $\omega\in
A_1$, define $(0,\infty] \ni c(\omega)= \liminf_{t\to\infty}
X(t,\omega)$ Then there exists $T_1(\omega)>0$ such that
$X(t,\omega)>0$ for all $t>T_1(\omega)$. Hence for $t\geq
T_1(\omega)$, we have
\begin{align*}
X(t)&=X(0)-\int_0^{T_1} f(X(s))\,ds - \int_{T_1}^t f(X(s))\,ds + \int_0^t \sigma(s)\,dB(s)\\
&\leq X(0)-\int_0^{T_1} f(X(s))\,ds + \int_0^t \sigma(s)\,dB(s).
\end{align*}
Since $\sigma\not\in L^2(0,\infty)$ it follows that
$\liminf_{t\to\infty} \int_0^t \sigma(s)\,dB(s)=-\infty$ a.s.
Therefore a.s. on $A_1$ we have
\[
c(\omega)=\liminf_{t\to\infty} X(t,\omega)\leq X(0)-\int_0^{T_1}
f(X(s))\,ds + \liminf_{t\to\infty} \int_0^t
\sigma(s)\,dB(s)=-\infty,
\]
a contradiction. Hence $\mathbb{P}[A_1]=0$, so
$\liminf_{t\to\infty} X(t)\leq 0$ a.s. To prove that
$\limsup_{t\to\infty} X(t)\geq 0$ a.s., define
$X_-(t)=-X(t)$, $f_-(x)=-f(-x)$, $\sigma_-(t)=-\sigma(t)$. Then
\[
dX_-(t)=-f_-(X_-(t))\,dt + \sigma_-(t)\,dB(t),\quad t\geq 0.
\]
By the same argument as above, it can be shown that
$\liminf_{t\to\infty} X_-(t)\leq 0$ a.s., which yields
$\limsup_{t\to\infty} X(t)\geq 0$ a.s. Combining this with
$\liminf_{t\to\infty} X(t)\leq 0$ a.s. yields the required result.

\section{Proofs of Proposition~\ref{prop.sigmasqlogtS} and \ref{prop.thetaSigma}} \label{sec.proofs1}
\subsection{Proof of Proposition~\ref{prop.sigmasqlogtS}}
By \eqref{eq.millsasy} we have
\[
\lim_{x\to\infty} \left\{\log(1-\Phi(x)) - \log x^{-1}
+x^2/2\right\}=\log\left(1/\sqrt{2\pi}\right),
\]
and so
\[
\lim_{x\to\infty} \frac{\log(1-\Phi(x))}{x^2/2}=-1.
\]
Suppose that $\theta(n)\to 0$ as $n\to\infty$, we have for
$\epsilon>0$ that
\[
\lim_{n\to\infty}
\frac{\log(1-\Phi(\epsilon/\theta(n)))}{\epsilon^2/(2\theta^2(n))}=-1.
\]
Thus
\begin{align}
\lim_{n\to\infty} \frac{\log(1-\Phi(\epsilon/\theta(n)))}{\log n}
&=\lim_{n\to\infty}
\frac{\log(1-\Phi(\epsilon/\theta(n)))}{\epsilon^2/(2\theta^2(n))}
\cdot \frac{\epsilon^2/(2\theta^2(n))}{\log n}\nonumber\\
\label{eq.limpolyexptail} &=-\frac{\epsilon^2}{2}\lim_{n\to\infty}
\frac{1}{\theta^2(n)\log n}.
\end{align}

In cases (A) and (B), we have that $\theta^2(n):=\int_{n}^{n+1}
\sigma^2(s)\,ds$ obeys
\begin{equation} \label{eq.limthetalogn}
\lim_{n\to\infty}\theta^2(n)\log n=L,
\end{equation}
and in each case $\theta(n)\to 0$ as $n\to\infty$. Therefore
\eqref{eq.limpolyexptail} holds in both case (A) and case (B). To
prove part (A), note that when $L=0$, from \eqref{eq.limthetalogn}
and \eqref{eq.limpolyexptail}, we have
\[
\lim_{n\to\infty} \frac{\log(1-\Phi(\epsilon/\theta(n)))}{\log
n}=-\infty
\]
for every $\epsilon>0$, so by \eqref{def.Sepsilon}, we have
$S(\epsilon)<+\infty$ for every $\epsilon>0$. Therefore, by
Lemma~\ref{lemma.sepss'eps}, $S'$ obeys \eqref{eq.thetastable}, as
required.

To prove part (B), note that when $L\in (0,\infty)$, from
\eqref{eq.limthetalogn} and \eqref{eq.limpolyexptail}, we have
\[
\lim_{n\to\infty} \frac{\log(1-\Phi(\epsilon/\theta(n)))}{\log
n}=-\frac{\epsilon^2}{2L}.
\]
If $\epsilon>\sqrt{2L}$, then by \eqref{def.Sepsilon} we have
$S(\epsilon)<+\infty$, and thus by Lemma~\ref{lemma.sepss'eps},
$S'(\epsilon)<+\infty$. On the other hand, if $\epsilon<\sqrt{2L}$,
by \eqref{def.Sepsilon} we have that $S(\epsilon)=+\infty$, and so
by Lemma~\ref{lemma.sepss'eps}, $S'(\epsilon)=+\infty$. Therefore
\eqref{eq.thetabounded} holds with $\epsilon'=\sqrt{2L}$.

In case (C), suppose that there exists $\epsilon^\ast>0$ such that
$S'(\epsilon^\ast)<+\infty$. Then by Lemma~\ref{lemma.sepss'eps}, we
have that $S(\epsilon^\ast)<+\infty$. Then we have that
$1-\Phi(\epsilon^\ast/\theta(n))\to 0$ as $n\to\infty$. This implies
that $\theta(n)\to 0$ as $n\to\infty$. Thus, we have that
\eqref{eq.limpolyexptail} holds. Now, because $\sigma^2(t)\log
t\to\infty$ as $t\to\infty$, we have that $\theta^2(n)\log
n\to\infty$ as $n\to\infty$. Therefore, using this fact and
\eqref{eq.limpolyexptail}, we have that
\[
\lim_{n\to\infty} \frac{\log(1-\Phi(\epsilon^\ast/\theta(n)))}{\log
n}=0.
\]
Therefore, it follows from \eqref{def.Sepsilon} that
$S(\epsilon^\ast)=+\infty$, a contradiction. Therefore, we must have
that $S'(\epsilon)=+\infty$ for every $\epsilon>0$, which is
\eqref{eq.thetaunstable}, as claimed.

\subsection{Proof of Proposition~\ref{prop.thetaSigma}}
For $n\leq t<n+1$, we have that $\Sigma^2(t)\leq e^2 \Sigma^2(n+1)$
and $\Sigma^2(t)\geq e^{-2}\Sigma^2(n)$. Thus it is easy to see that
$\Sigma^2(t)\to 0$ as $t\to\infty$ if and only if the sequence
$(\Sigma^2(n))_{n\geq 0}$ converges to zero.

Writing
\[
\Sigma^2(n)=e^{-2n}\sum_{j=0}^{n-1} \int_{j}^{j+1}
e^{2s}\sigma^2(s)\,ds \cdot \log n,
\]
and using \eqref{def.capTheta} we readily get the double inequality
\begin{equation} \label{eq.thetaTheta}
\Theta^2(n) \cdot \log n \leq \Sigma^2(n)\leq e^2  \Theta^2(n) \cdot
\log n.
\end{equation}
Hence by considering the last term in the sum on the left hand side
of \eqref{eq.thetaTheta}, we get $\Sigma^2(n)\geq e^{-2}
\theta^2(n-1)\log n$, so $\theta^2(n-1) \log (n-1) \leq e^2
\Sigma^2(n)$. Therefore, if $\Sigma^2(n)\to 0$, we have that
$\theta^2(n)\log n \to 0$ as $n\to\infty$.

On the other hand, if $\theta^2(n)\log n \to 0$ as $n\to\infty$, for
every $\epsilon>0$ there exists an integer $N(\epsilon)\geq 1$ such
that $\theta^2(n)\log n<\epsilon$ for all $n\geq N(\epsilon)$. Thus
for $n\geq N(\epsilon)+1$, by \eqref{eq.thetaTheta}, we have
\[
\Sigma^2(n)
\leq e^2 \frac{\sum_{j=0}^{N(\epsilon)-1} e^{2j}
\theta^2(j)}{e^{2n}/\log n} + \epsilon e^2
\frac{\sum_{j=N(\epsilon)}^{n-1} e^{2j}/\log j}{e^{2n}/\log n},
\]
so
\[
\limsup_{n\to\infty}  \Sigma^2(n) \leq \epsilon e^2
\limsup_{n\to\infty}\frac{\sum_{j=2}^{n-1} e^{2j}/\log
j}{e^{2n}/\log n}.
\]
Since $x\mapsto e^{2x}/\log x$ is increasing on $[2,\infty)$ we have
that
\[
\sum_{j=2}^{n-1} e^{2j}/\log j \leq \sum_{j=2}^{n-1} \int_{j}^{j+1}
e^{2x}/\log x\,dx = \int_2^n e^{2x}/\log x\,dx
\]
By l'H\^{o}pital's rule
\[
\lim_{t\to\infty} \frac{\int_2^t e^{2x}/\log x\,dx}{e^{2t}/\log t} =
\lim_{t\to\infty} \frac{1}{2-1/(t\log t)}=\frac{1}{2},
\]
so
\[
 \limsup_{n\to\infty}\frac{\sum_{j=2}^{n-1} e^{2j}/\log j}{e^{2n}/\log n}\leq \frac{1}{2}.
\]
Hence $\limsup_{n\to\infty}  \Sigma^2(n) \leq \epsilon e^2/2$. Since
$\epsilon>0$ is arbitrary, we have $\Sigma^2(n)\to 0$ as
$n\to\infty$, as required.

Since for $t\in [n,n+1)$ we have $\Sigma^2(t)\leq e^2 \Sigma^2(n+1)$
and $\Sigma^2(t)\geq e^{-2}\Sigma^2(n)$, it follows that
$\limsup_{t\to\infty} \Sigma^2(t)\in (0,\infty)$ implies
$\limsup_{n\to\infty} \Sigma^2(n)<+\infty$. If $\limsup_{n\to\infty}
\Sigma^2(n)=0$, then $\limsup_{t\to\infty} \Sigma^2(t)=0$, a
contradiction. Therefore we have $\limsup_{t\to\infty}
\Sigma^2(t)\in (0,\infty)$ implies $\limsup_{n\to\infty}
\Sigma^2(n)\in (0,\infty)$. On the other hand, if
$\limsup_{n\to\infty} \Sigma^2(n)=:L\in (0,\infty)$, we see
immediately that
\[
0<e^{-2} L\leq \limsup_{t\to\infty} \Sigma^2(t) \leq e^2 L<+\infty,
\]
so the first two statements are equivalent. By
\eqref{eq.thetaTheta}, we have
$\limsup_{n\to\infty}\Sigma^2(n)=L\in (0,\infty)$ implies that
\[
\limsup_{n\to\infty} \Theta^2(n) \log n \leq L, \quad
\limsup_{n\to\infty} \Theta^2(n)\log n \geq L/e^2.
\]
By \eqref{def.capTheta}, it follows that
$\theta^2(n)=e^2\Theta^2(n+1)-\Theta^2(n)$. Hence
$\limsup_{n\to\infty}\Sigma^2(n)=L\in (0,\infty)$ implies that
\[
\limsup_{n\to\infty} \theta^2(n)\log n \leq
e^2\limsup_{n\to\infty}\Theta^2(n+1)\log n
+\limsup_{n\to\infty}\Theta^2(n)\log n \leq (1+e^2)L.
\]
To see that $\limsup_{n\to\infty} \theta^2(n)\log n>0$, suppose not.
Then $\theta^2(n)\log n\to 0$ as $n\to\infty$, and thus by part (i),
we have that $\Sigma^2(n)\to 0$ as $n\to\infty$, which is false by
hypothesis. Thus $\limsup_{n\to\infty}\Sigma^2(n)=L\in (0,\infty)$
implies $\limsup_{n\to\infty} \theta^2(n)\log n \in (0,\infty)$.

If $\limsup_{n\to\infty}\theta^2(n)\log n=:L'\in (0,\infty)$, then
$L'':=\limsup_{n\to\infty} \Theta^2(n)\log n\leq L'e^2/2$. Suppose
that $L''=0$. Then $\Sigma^2(n)\to 0$ as $n\to\infty$. But this
implies that $\theta^2(n)\log n\to 0$ as $n\to\infty$, a
contradiction. Therefore we have that $\limsup_{n\to\infty}
\Theta^2(n)\log n\in (0,\infty)$, proving the last equivalence.

To prove part (iii), note that for $n\leq t<n+1$, we have
$e^{-2}\Sigma^2(n)\leq \Sigma^2(t)\leq e^2\Sigma^2(n+1)$. Hence
$\liminf_{t\to \infty}\Sigma^2(t)\in (0,\infty)$ implies that
$\liminf_{n\to \infty}\Sigma^2(n)>0$. To see $\liminf_{n\to
\infty}\Sigma^2(n)<+\infty$, suppose not, then $\liminf_{t\to
\infty}\Sigma^2(t)=+\infty$, arriving at a contradiction. Thus we
must have that $\liminf_{t\to \infty}\Sigma^2(t)\in (0,\infty)$
implies $\liminf_{n\to \infty}\Sigma^2(n)\in (0,\infty)$

Let $\liminf_{n\to \infty}\Sigma^2(n)\in(0,\infty)=L$, then
\[0<e^{-2}L\leq \liminf_{t\to \infty}\Sigma^2(t)\leq e^2L<\infty\]
i.e. $\liminf_{t\to \infty}\Sigma^2(t)\in (0,\infty)$. Thus (A) and
(B) are equivalent.


By \eqref{eq.thetaTheta}, $\liminf_{n\to \infty}\Sigma^2(n)\in
(0,\infty)$ if and only if $\liminf_{n\to \infty}\Theta^2(n)\log
n\in (0,\infty)$. Also $\Sigma^2(n)\geq e^{-2}\theta^2(n-1)\log n$
implies that $ \liminf_{n\to \infty}\theta^2(n-1)\log (n-1)\leq
\liminf_{n\to \infty}e^{2}\Sigma^2(n)<\infty$, with $\liminf_{n\to
\infty}\theta^2(n)\log n\geq 0$ by hypothesis, we finish the proof
for (iii).

To prove (iv), note that for $n\leq t<n+1$, we have
$e^{-2}\Sigma^2(n)\leq \Sigma^2(t)\leq e^2\Sigma^2(n+1)$. Hence
$\lim_{t\to \infty}\Sigma^2(t)=\infty$ if and only if $\lim_{n\to
\infty}\Sigma^2(n)=\infty$. Again from \eqref{eq.thetaTheta}, we
have $\lim_{n\to \infty}\Theta^2(n)\log n=\infty$ implies
$\lim_{n\to \infty}\Sigma^2(n)=\infty$ and vice versa. To prove that
all imply that $\limsup_{n\to \infty}\theta^2(n)\log n=\infty$,
suppose $\limsup_{n\to \infty}\theta^2(n)\log n=L^{'}<\infty$, then
from $\theta^2(n)\log n=e^2\Theta^2(n+1)-\Theta^2(n)$, we have
$\limsup_{n\to \infty}\theta^2(n)\log n = \limsup_{n\to
\infty}[e^2\Theta^2(n+1)-\Theta^2(n)]<\infty$. Therefore
$\limsup_{n\to \infty}\Theta^2(n)\log n<\infty$, which is a
contradiction.

\section{Proof of Proposition~\ref{prop.yimpliesx}} \label{sec:s9}

\subsection{Proof of Part (A) of Proposition~\ref{prop.yimpliesx}} 
In the case when $\sigma\in L^2(0,\infty)$, we have that each of the
events $\{\omega:\lim_{t\to\infty} Y(t,\omega)=0\}$ and
$\{\omega:\lim_{t\to\infty} X(t,\omega)=0\}$ are a.s. by
Theorem~\ref{th.sigL2}.

Suppose now that $\sigma\not \in L^2(0,\infty)$. Define
\begin{equation} \label{def.Omegae}
\Omega_e=\Omega_X\cap\Omega_Y,
 \end{equation}
 where $\Omega_X$ is given by \eqref{def.OmegaX} and $\Omega_Y$ is defined by \eqref{def.OmegaY}. Define for each $\omega\in
\Omega_e$ the realisation $z(\cdot,\omega)$ by
$z(t,\omega)=X(t,\omega)-Y(t,\omega)$ for $t\geq 0$. Then
$z(\cdot,\omega)$ is in $C^1(0,\infty)$ and obeys
\begin{equation*}
z'(t,\omega)=-f(X(t,\omega))+Y(t,\omega)=-f(z(t,\omega)+Y(t,\omega))+Y(t,\omega),
\, t\geq 0; \, z(0)=\xi.
\end{equation*}

Define
\[
A_2=\{\omega\in \Omega_e:\lim_{t\to\infty} Y(t,\omega)=0\},
\quad A_3=\{\omega\in \Omega_e: \liminf_{t\to\infty}
|X(t,\omega)|=0\}.
\]
Therefore $A_2$ is an a.s. event by hypothesis. Since
$\sigma\not\in L^2(0,\infty)$, $A_3$ is an a.s. event by
Theorem~\ref{lemma.Xliminf}. Thus the event $A_4$ defined by $A_4=A_2\cap A_3$ is almost sure.
Fix $\omega\in A_4$. Since $Y(t,\omega)\to 0$ as $t\to\infty$
and $\liminf_{t\to\infty} |X(t,\omega)|=0$, it follows that
\[
\liminf_{t\to\infty} |z(t,\omega)|\leq \liminf_{t\to\infty}
|X(t,\omega)|+|Y(t,\omega)|=\liminf_{t\to\infty}
|X(t,\omega)|+\lim_{t\to\infty} |Y(t,\omega)|=0.
\]
Let $\eta\in (0,1)$. We next show that $\limsup_{t\to\infty}
|z(t,\omega)|\leq \eta$. Since $f$ is continuous on $\mathbb{R}$, it is uniformly continuous on $[-2,2]$. Therefore,
there exists a function $\mu:[0,\infty)\to [0,\infty)$ such that $\mu(0)=0$, $\mu(\nu)\to 0$ as $\nu\downarrow 0$, and for which for every $\nu\in [0,4]$ is defined by
\[
\mu(\nu)=\max_{|x|\vee|y|\leq 2,|x-y|\leq \nu}|f(x)-f(y)|.
\]
Thus $\mu$ is a modulus of continuity of $f$ on $[-2,2]$. Let $\epsilon>0$ be so small that
\[
\epsilon<\frac{\eta}{4}, \quad \epsilon+\mu(\epsilon)<f(\eta)\wedge|f(-\eta)|.
\]
Then for $u\in [\eta-\epsilon,\eta+\epsilon] \subset(0,2)$ we have
$|f(u)-f(\eta)|\leq \mu(\epsilon)$, so $f(u)\geq f(\eta)-\mu(\epsilon)>\epsilon$. Therefore
\begin{equation} \label{eq.f1}
\epsilon<\inf_{u\in (\eta-\epsilon,\eta+\epsilon)} f(u).
\end{equation}
On the other hand for $u\in [-\eta-\epsilon,-\eta+\epsilon]\subset
(-2,0)$ we have $|f(u)-f(-\eta)|\leq \mu(\epsilon)$, so
\[
f(u)\leq f(-\eta)+\mu(\epsilon)<-\epsilon.
\]
Therefore
\begin{equation} \label{eq.f2}
-\epsilon>\sup_{u\in (\eta-\epsilon,\eta+\epsilon)} f(u).
\end{equation}

Since $Y(t,\omega)\to 0$ as $t\to\infty$, there exists
$T_1(\epsilon,\omega)>0$ such that $|Y(t,\omega)|<\epsilon$ for all
$t>T_1(\epsilon)$. Suppose that $\limsup_{t\to\infty}
|z(t,\omega)|>\eta$. Since $\liminf_{t\to\infty} |z(t,\omega)|=0$,
we may therefore define
$T_2(\epsilon,\omega)=\inf\{t>T_1(\epsilon,\omega):|z(t,\omega)|=\eta/2\}$.
Also define
$T_3(\epsilon,\omega)=\inf\{t>T_2(\epsilon,\omega):|z(t,\omega)|=\eta\}$.

In the case when $z(T_3(\epsilon,\omega),\omega)=\eta$, we have that
$z'(T_3(\epsilon,\omega),\omega)\geq 0$. Since
$|Y(T_3(\epsilon,\omega),\omega)|< \epsilon$ we have
\begin{align*}
0&\leq z'(T_3(\epsilon,\omega),\omega)=-f(z(T_3(\epsilon,\omega))+Y(T_3(\epsilon,\omega),\omega))+Y(T_3(\epsilon,\omega),\omega)\\
&=-f(\eta+Y(T_3(\epsilon,\omega),\omega))+Y(T_3(\epsilon,\omega),\omega)\\
&<-f(\eta+Y(T_3(\epsilon,\omega),\omega))+\epsilon \leq
-\inf_{|u-\eta|<\epsilon} f(u) + \epsilon<0,
\end{align*}
by \eqref{eq.f1}, a contradiction. On the other hand, in the case
when $z(T_3(\epsilon,\omega),\omega))=-\eta$, we have that
$z'(T_3(\epsilon,\omega),\omega)\leq 0$. Since
$|Y(T_3(\epsilon,\omega),\omega)|< \epsilon$ we have
\begin{align*}
0&\geq
z'(T_3(\epsilon,\omega),\omega)=-f(z(T_3(\epsilon,\omega),\omega)+Y(T_3(\epsilon,\omega),\omega))
+Y(T_3(\epsilon,\omega))\\
&=-f(-\eta+Y(T_3(\epsilon,\omega),\omega))+Y(T_3(\epsilon,\omega),\omega)\\
&>-f(-\eta+Y(T_3(\epsilon,\omega),\omega))-\epsilon \geq
-\sup_{|u+\eta|<\epsilon} f(u) - \epsilon>0,
\end{align*}
by \eqref{eq.f2}, a contradiction. Hence $T_3(\epsilon,\omega)$ does
not exist for any $\omega\in A_4$. Hence $\limsup_{t\to\infty}
|z(t,\omega)|\leq \eta$. Since $\eta>0$ is arbitrary, we make take
the limit as $\eta\downarrow 0$ to obtain $\limsup_{t\to\infty}
|z(t,\omega)|=0$. Since $X=Y+z$, and $Y(t,\omega)\to 0$ as
$t\to\infty$, we have that $X(t,\omega)\to 0$ as $t\to\infty$, and because this is true for each $\omega$ in the a.s. event $A_4$, the result has been proven.

\subsection{Proof of Part (C) of Proposition~\ref{prop.yimpliesx}}
Let the a.s. event $\Omega_e$ be as defined in \eqref{def.Omegae}.
Define $\Omega_3=\{\omega\in \Omega_e: \limsup_{t\to
\infty}|Y(t,\omega)|=+\infty\}$ which is a.s. by hypothesis. Define
$F(t)=X(t)-f(X(t))$ for $t\geq 0$. Then \eqref{eq.sde} can be
rewritten as
\[
dX(t)=\left\{-X(t)+F(t)\right\}\,dt+\sigma(t) \,dB(t), \quad t\geq
0,
\]
so by variation of constants we get
\[
X(t)=X(0)e^{-t}+\int^{t}_{0}e^{-(t-s)}F(s)ds+ Y(t), \quad t\geq 0.
\]
Rearranging and taking absolute values gives
\begin{equation} \label{eq.estY}
|Y(t)|\leq |X(t)|+|X(0)|e^{-t}+\int^{t}_{0}e^{-(t-s)}|F(s)|\,ds,
\quad t\geq 0.
\end{equation}
Define $A_5=\left\{\omega\in \Omega_X: \sup_{t\geq
0}|X(t,\omega)|<+\infty \right\}$ and suppose that $\mathbb{P}[A_5]>0$.
Define $A_6=A_5\cap \Omega_3$. Then
$\mathbb{P}[A_6]=\mathbb{P}[A_5]>0$. Let $\omega \in A_6$ and
define $X_1(\omega)=\sup_{t\geq 0}|X(t,\omega)|$. Then
$|X(t,\omega)|\leq X_1(\omega)$ for all $t\geq 0$. Since
$f$ is continuous, for all $y\geq 0$, there exists
$\overline{f}(y)<+\infty$ such that
\begin{equation} \label{eq.Fbar}
\max_{|x|\leq y}|f(x)|=:\overline{f}(y).
\end{equation}
Therefore $|f(X(t,\omega))|\leq \overline{f}(X_1(\omega))$ for all
$t\geq 0$. Hence by \eqref{eq.estY},  for each $\omega\in A_6$, we
have that for all $t\geq 0$
\begin{align*}
|Y(t,\omega)|&\leq X_1(\omega)+X_1(\omega)+\int^{t}_{0}e^{-(t-s)}(X_1(\omega)+\overline{f}(X_1(\omega)))\,ds\\
&\leq 3X_1(\omega)+\overline{f}(X_1(\omega)).
\end{align*}
Since $\limsup_{t\to \infty}|Y(t,\omega)|=+\infty$ for each
$\omega\in A_6\subseteq \Omega_3$, we have a contradiction, so
therefore we must have $\mathbb{P}[A_6]=0$. This, taken together with
continuity the continuity of $X$, gives
$\limsup_{t\to\infty}|X(t)|=\infty$ a.s., proving part (C) of
Proposition~\ref{prop.yimpliesx}.

\subsection{Proof of Part (B) of Proposition~\ref{prop.yimpliesx}}
Define $\Omega_2=\Omega_1\cap \Omega_e$. Then by hypothesis, for every $\omega\in \Omega_2$ we have
that there is a finite and positive $Y^\ast(\omega)$ such that
\[
Y^\ast(\omega)=\limsup_{t\to\infty} |Y(t,\omega)|.
\]
By definition $\underline{Y}\leq Y^\ast(\omega)\leq \overline{Y}$. Define for $\omega\in \Omega_2$
\[
X^\ast(\omega)=\limsup_{t\to\infty} |X(t,\omega)|,
\]
where $X^\ast(\omega)=0$ and $X^\ast(\omega)=+\infty$ are admissible values.
By \eqref{eq.estY}, we have
\begin{equation*}
Y^\ast(\omega) \leq X^\ast(\omega)+\limsup_{t\to\infty}\int^{t}_{0}e^{-(t-s)}|F(s,\omega)|\,ds
\leq X^\ast(\omega)+\limsup_{t\to\infty} |F(t,\omega)|.
\end{equation*}
By the definition of $\overline{f}$, $F$ and $X^\ast$ we have
\[
\limsup_{t\to\infty}  |F(t,\omega)| \leq X^\ast(\omega)+\overline{f}(X^\ast(\omega)).
\]
Since $\overline{f}$ is defined by \eqref{eq.Fbar} and $h_f$ by \eqref{def.hf}, we obtain
\[
Y^\ast(\omega)\leq 2X^\ast(\omega)+\overline{f}(X^\ast(\omega))=h_f(X^\ast(\omega)).
\]
By Proposition~\ref{prop.underline}, $h_f$ is an increasing
function, so we have $X^\ast(\omega)\geq h_f^{-1}(Y^\ast(\omega))$.
Now by the definition of $X^\ast$, $\underline{X}$ and the fact that
$h_f^{-1}$ is increasing, we have
\[
\underline{X}=\inf_{\omega\in \Omega_2} X^\ast(\omega)\geq \inf_{\omega\in \Omega_2} h_f^{-1}(Y^\ast(\omega))
= h_f^{-1}\left(\inf_{\omega\in \Omega_2} Y^\ast(\omega)\right).
\]
Since $\Omega_2\subseteq \Omega_1$, $\inf_{\omega\in \Omega_2}
Y^\ast(\omega)\geq \inf_{\omega\in \Omega_1}
Y^\ast(\omega)=\underline{Y}$, by the definition of $\underline{Y}$.
Thus as $h_f^{-1}$ is increasing,
\[
\underline{X}\geq  h_f^{-1}\left(\inf_{\omega\in \Omega_2} Y^\ast(\omega)\right)
\geq  h_f^{-1}(\underline{Y})=\underline{x}(f,\underline{Y}),
\]
using \eqref{eq.underxhfinv} at the last step. Notice lastly that part (i) of Proposition~\ref{prop.underline} implies that $\underline{x}(f,\underline{Y})>0$ because $\underline{Y}>0$, by hypothesis.
\section{Proofs of Theorem~\ref{theorem.Xbounded} and
Proposition~\ref{prop.overline}}   \label{sec:s10}
\subsection{Preliminary results}
The asymptotic estimate \eqref{eq.limsupXuprboundfinal} in Theorem~\ref{theorem.Xbounded} is shown by first establishing the estimate
\begin{equation} \label{eq.limsupXuprbound1}
\limsup_{t\to\infty} |X(t,\omega)|\leq
\max(x_+(\overline{Y}),x_-(\overline{Y}))+\overline{Y},
\quad\text{for each $\omega\in \Omega_2$}
\end{equation}
where we define $x_+,x_-:[0,\infty)\to \mathbb{R}$  by
\begin{align} \label{def.xplus}
x_+(y)&=\sup\{x>0: \min_{a\in [-y,y]} f(x+a)=y\},\quad y\geq 0,\\
\label{def.xminus} -x_-(y)&=\inf\{x<0: \max_{a\in [-y,y]}
f(x+a)= -y\}, \quad y\geq 0.
\end{align}
We prefer the estimate in \eqref{eq.limsupXuprboundfinal} in part
because the estimate on the right hand side of
\eqref{eq.limsupXuprbound1} is difficult to analyse in general, due
to the complexity of $x_+$ and $x_-$. Moreover, there is no loss of
sharpness in the estimate in \eqref{eq.limsupXuprboundfinal}
relative to \eqref{eq.limsupXuprbound1} in the case when $f$ is
increasing. To see this, first note that when $f$ is increasing on
$\mathbb{R}$, it can readily be seen that $x_+(y)=y+f^{-1}(y)$ and
$x_-(y)=y-f^{-1}(-y)$. Therefore, if we grant that
\eqref{eq.limsupXuprbound1} holds, it follows that
\[
\limsup_{t\to\infty} |X(t,\omega)|\leq
2\overline{Y}+\max(f^{-1}(\overline{Y}),-f^{-1}(\overline{Y})),
\quad\text{for each $\omega\in \Omega_2$.}
\]
Therefore, if we define
\begin{equation}  \label{eq.overlineXfinv}
\overline{x}^\ast(f,y) =2y +
\max(f^{-1}(y),-f^{-1}(-y)),
\end{equation}
it can be seen that
\[
\limsup_{t\to\infty} |X(t,\omega)| \leq
\overline{x}^\ast(f,\overline{Y}), \quad \text{for each $\omega\in
\Omega_2$.}
\]
On the other hand, $\overline{x}^\ast(f)$ defined in
\eqref{eq.overlineXfinv} is equal to $\overline{x}(f)$ defined in
\eqref{def.overlineX} when $f$ is increasing, because
$f^-(x)=f^{-1}(x)$ for $x\leq 0$ and $f^+(x)=f^{-1}(x)$ for $x\geq
0$, where $f^+$ and $f^-$ are defined in \eqref{def.fgeninv} and
\eqref{def.fgeninvneg}.

Therefore, the second stage in proving the asymptotic estimate
\eqref{eq.limsupXuprboundfinal} reduces to showing that
\begin{equation} \label{eq.xpxmovx}
y+\max(x_+(y),x_-(y))\leq \overline{x}(f,y), \quad y\geq 0,
\end{equation}
and accordingly, we start the proof of Theorem~\ref{theorem.Xbounded} by first establishing \eqref{eq.xpxmovx}.
\begin{lemma} \label{prop.upperboundX}
Suppose that $f$ obeys \eqref{eq.fglobalunperturbed} and \eqref{h.finfty}. Then the functions $f^+$
and $f^-$ given by \eqref{def.fgeninv} and \eqref{def.fgeninvneg}
are well--defined and with $x_+$, $x_-$ and $\overline{x}$ defined by \eqref{def.xplus}, \eqref{def.xminus}
and \eqref{def.overlineX} respectively, we have \eqref{eq.xpxmovx}.
%
\end{lemma}
\begin{proof}
Let $z>x+f^+(x)$. Suppose $u\in [-x,x]$. Then $z+u>f^+(x)$. By the
definition of $f^+$ we have $f(a)>x$ for all $a>f^+(x)$. Therefore,
for each $z>x+f^+(x)$, we have $f(z+u)>x$ for all $u\in [-x,x]$.
Hence
\[
\min_{u\in [-x,x]} f(z+u)>x, \quad \text{for all $z>x+f^+(x)$}.
\]
Since $x_+(y)=\sup\{x>0: \min_{u\in [-y,y]} f(x+u)=y\}$, we have that
\begin{equation} \label{eq.iplusineq}
y+f^+(y)\geq x_+(y).
\end{equation}

Let $x>0$. Let $z<-x+f^-(-x)$. Suppose $u\in [-x,x]$. Then
$z+u<f^-(-x)$. By the definition of $f^-$ we have $f(a)<-x$ for all
$a<f^-(-x)$. Therefore, for each $z<-x+f^-(-x)$, we have $f(z+u)<-x$
for all $u\in [-x,x]$. Hence
\[
\max_{u\in [-x,x]} f(z+u)<-x, \quad \text{for all $z<-x+f^-(-x)$}.
\]
Since $-x_-(y)=\inf\{x>0: \max_{u\in [-y,y]}
f(x+u)=-y\}$, we have that
$-y+f^-(-y)\leq -x_-(y)$, so
\begin{equation} \label{eq.iminusineq}
y-f^-(-y)\geq x_-(y).
\end{equation}
Hence by \eqref{def.overlineX}, \eqref{eq.iplusineq}, \eqref{eq.iminusineq}, for any $y\geq 0$ we have
\begin{align*}
\overline{x}(f,y)&=2y+\max(f^+(y),-f^-(-y))\\
&=y+\max(y+f^+(y),y-f^-(-y))\\
&\geq y+\max(x_+(y),x_-(y)),
\end{align*}
which is \eqref{eq.xpxmovx}.
\end{proof}

\subsection{Proof of Theorem~\ref{theorem.Xbounded}}
We start with a lemma.
\begin{lemma}   \label{lemma.odelimsup}
Let $f$ obey \eqref{eq.fglobalunperturbed} and \eqref{h.finfty}. Suppose that $p$ is a continuous function
such that
\[
\limsup_{t\to\infty} |p(t)|\leq \overline{p}.
\]
Suppose that $z$ is any continuous solution of
\[
z'(t)=-f(z(t)+p(t))+p(t), \quad t>0; \quad z(0)=\xi
\]
Then
\[
\limsup_{t\to\infty} |z(t)|\leq \max(x_+(\overline{p}),x_-(\overline{p}))\leq \overline{p}+\max(f^+(\overline{p}),-f^-(-\overline{p}),
\]
where $x_+$ is defined by \eqref{def.xplus}, $x_-$ by \eqref{def.xminus} and $f^\pm$ by \eqref{def.fgeninv}, \eqref{def.fgeninvneg}. Moreover, if $x(t)=z(t)+p(t)$ for $t\geq 0$,
and $\overline{x}$ is defined by \eqref{def.overlineX}, then
\[
\limsup_{t\to\infty} |x(t)| \leq \overline{x}(f,\overline{p}).
\]
\end{lemma}
\begin{proof}
For every $\eta>0$, there exists $T(\eta)>0$ such that for $t\geq T(\eta)$ we have $|p(t)|\leq \overline{p}+\eta$.

The bound on $p$ yields the estimate
\[
z(t)-\overline{p}-\eta \leq z(t)+p(t)\leq z(t)+\overline{p}+\eta,
\quad t\geq T(\eta).
\]
Since $f(x)\to\infty$ as $x\to\infty$, for every $\eta>0$ there
exists $\tilde{x}_+(\eta)>\eta$ such that
\[
\min_{a\in [-\overline{p}-\eta,\overline{p}+\eta]} f(x+a)\geq
\overline{p}+2\eta, \quad \text{for all $x\geq \tilde{x}_+(\eta)$}.
\]
Note that $x_+$ defined by \eqref{def.xplus} obeys
\begin{equation} \label{eq.xplus}
\min_{a\in [-\overline{p},\overline{p}]} f(x+a)\geq \overline{p},
\quad \text{for all $x\geq x_+(\overline{p})$}.
\end{equation}
Also as $f(x)\to-\infty$ as $x\to-\infty$, for every $\eta>0$ there
exists an $\tilde{x}_-(\eta)>\eta$ such that
\[
\max_{a\in [-\overline{p}-\eta,\overline{p}+\eta]} f(x+a)\leq
-\overline{p}-2\eta, \quad\text{for all $x\leq -\tilde{x}_-(\eta)$}.
\]
Note that $x_-$ defined by \eqref{def.xplus} obeys
\begin{equation} \label{eq.xminus}
\max_{a\in [-\overline{p},\overline{p}]} f(x+a)\leq -\overline{p},
\quad\text{for all $x\leq -x_-(\overline{p})$}.
\end{equation}
Let $x(\eta)=\max(\tilde{x}_+(\eta),\tilde{x}_-(\eta))$.

Suppose that there is $t_1(\eta)>T(\eta)$ such that
$z(t_1)>\tilde{x}_+(\eta)$. If not, it follows that
\[
z(t)\leq \tilde{x}_+(\eta) \text{ for all $t\geq T(\eta)$}
\]
and we have that $\limsup_{t\to\infty} z(t)\leq \tilde{x}_+(\eta)$,
which implies that $\limsup_{t\to\infty} z(t)\leq
x_+(\overline{p})$. We will show that there exists a
$t_2(\eta)>t_1(\eta)$ such that $z(t_2)=\tilde{x}_+(\eta)$ and
moreover for all $t\geq t_2(\eta)$ that $z(t)\leq
\tilde{x}_+(\eta)$. This implies that $\limsup_{t\to\infty} z(t)\leq
\tilde{x}_+(\eta)$ or indeed that $\limsup_{t\to\infty} z(t)\leq
x_+(\overline{p})$.

By the definition of $t_1$ we have $z(t_1)+p(t_1)>0$ and
\[
z'(t_1)=-f(z(t_1)+p(t_1))+p(t_1)\leq -\min_{a\in
[-\overline{p}-\eta,\overline{p}+\eta]} f(z(t_1)+a) +
\overline{p}+\eta\leq -\eta.
\]
Then we have either that $z(t)>\tilde{x}_+(\eta)$ for all $t\geq
t_1(\eta)$ or that there is a minimal $t_2(\eta)>t_1(\eta)$ such
that $z(t_2)=\tilde{x}_+(\eta)$. In the former case for every $t\geq
t_1(\eta)$ we have
\[
z'(t)=-f(z(t)+p(t))+p(t)\leq -\min_{a\in
[-\overline{p}-\eta,\overline{p}+\eta]} f(z(t)+a) +
\overline{p}+\eta\leq -\eta.
\]
Since $\eta>0$, we may define $t_3=(z(t_1)-\tilde{x}_+(\eta))/\eta +
t_1+1$. Then $z(t_3)\leq z(t_1)-\eta(t_3-t_1)<\tilde{x}_+(\eta)$, a
contradiction. Therefore, there exists a $t_2>t_1$ such that
$z(t_2)=\tilde{x}_+(\eta)$. Now
\[
z'(t_2)=-f(z(t_2)+p(t_2))+p(t_2)\leq -\min_{a\in [-\overline{p}-\eta,\overline{p}+\eta]} f(\tilde{x}_+(\eta)+a)
+\overline{p}+\eta\leq -\eta.
\]
Then either there exists a minimal $t_3(\eta)>t_2(\eta)$ such that
$z(t_3)=\tilde{x}_+(\eta)$ or we have that $z(t)< \tilde{x}_+(\eta)$
for all $t> t_2(\eta)$. In the former case, we must have
$z'(t_3)\geq 0$. But once again we have
\[
z'(t_3)=-f(z(t_3)+p(t_3))+p(t_3)\leq -\min_{a\in[-\overline{p}-\eta,\overline{p}+\eta]} f(\tilde{x}_+(\eta)+a) +
\overline{p}+\eta\leq -\eta,
\]
a contradiction. Thus we have $z(t)< \tilde{x}_+(\eta)$ for all $t>
t_2(\eta)$, which implies that $\limsup_{t\to\infty} z(t)\leq
x_+(\overline{p})$.

Suppose that there is $t_1(\eta)>T(\eta)$ such that $z(t_1)<-\tilde{x}_-(\eta)$. If not, it follows that
\[
z(t)\geq -\tilde{x}_-(\eta) \text{ for all $t\geq T(\eta)$}
\]
and we have that $\liminf_{t\to\infty} z(t)\geq -\tilde{x}_-(\eta)$,
which implies that $\liminf_{t\to\infty} z(t)\geq -x_-(\overline{p})$. We will show that there is a
$t_2(\eta)>t_1(\eta)$ such that $z(t_2)=-\tilde{x}_-(\eta)$ and moreover that for all $t\geq t_2(\eta)$ that
$z(t)\geq -\tilde{x}_-(\eta)$. This will imply that $\liminf_{t\to\infty} z(t)\geq -\tilde{x}_-(\eta)$ or that $\liminf_{t\to\infty} z(t)\geq -x_-(\overline{p})$.

By the definition of $t_1$ we have $z(t_1)+p(t_1)<0$ and
\[
z'(t_1)=-f(z(t_1)+p(t_1))+p(t_1)
\geq -\max_{a\in [-\overline{p}-\eta,\overline{p}+\eta]} f(z(t_1)+a) -\overline{p}-\eta\geq \eta.
\]
Then we have either that $z(t)<-\tilde{x}_-(\eta)$ for all $t\geq
t_1(\eta)$ or that there is a minimal $t_2(\eta)>t_1(\eta)$ such
that $z(t_2)=-\tilde{x}_-(\eta)$. In the former case for every
$t\geq t_1(\eta)$ we have
\[
z'(t)=-f(z(t)+p(t))+p(t)
\geq -\max_{a\in [-\overline{p}-\eta,\overline{p}+\eta]} f(z(t)+a) -\overline{p}-\eta\geq \eta.
\]
Since $\eta>0$, we may define $t_3=(z(t_1)+\tilde{x}_-(\eta))/-\eta
+ t_1+1$. Then
$z(t_3)\geq z(t_1)+\eta(t_3-t_1)>-\tilde{x}_-(\eta)$, a contradiction. Therefore, there exists a $t_2>t_1$ such that
$z(t_2)=-\tilde{x}_-(\eta)$. Now
\[
z'(t_2)=-f(z(t_2)+p(t_2))+p(t_2)
\geq -\max_{a\in [-\overline{p}-\eta,\overline{p}+\eta]} f(\tilde{x}_+(\eta)+a) -\overline{p}-\eta\geq \eta.
\]
Then either there exists a minimal $t_3(\eta)>t_2(\eta)$ such that
$z(t_3)=-\tilde{x}_-(\eta)$ or we have that $z(t)>-\tilde{x}_-(\eta)$ for all $t> t_2(\eta)$. In the former case, we
must have $z'(t_3)\leq 0$. But once again we have
\[
z'(t_3)=-f(z(t_3)+p(t_3))+p(t_3)
\geq -\max_{a\in [-\overline{p}-\eta,\overline{p}+\eta]} f(\tilde{x}_+(\eta)+a) -\overline{p}-\eta\geq \eta,
\]
a contradiction. Thus we have $z(t)> -\tilde{x}_-(\eta)$ for all $t>
t_2(\eta)$, which implies that $\liminf_{t\to\infty} z(t)\geq
-x_-(\overline{p})$.

We have thus shown that
\[
\limsup_{t\to\infty} z(t)\leq x_+(\overline{p}), \quad
\liminf_{t\to\infty} z(t)\geq -x_-(\overline{p}),
\]
and so $\limsup_{t\to\infty} |z(t)|\leq \max(x_+(\overline{p}),x_-(\overline{p}))$, as required.

Since $\limsup_{t\to\infty} |p(t)|\leq \overline{p}$, it follows that
\[
\limsup_{t\to\infty} |x(t)| \leq \overline{p}+\max(x_+(\overline{p}),x_-(\overline{p})).
\]
Therefore using Lemma~\ref{prop.upperboundX} (specifically
\eqref{eq.xpxmovx}), we have
\[
\limsup_{t\to\infty} |x(t)|\leq \overline{p}+\max(x_+(\overline{p}),x_-(\overline{p}))\leq
\overline{x}(f,\overline{p}),
\]
which is precisely the final estimate required. 
\end{proof}

\subsection{Proof of Theorem~\ref{theorem.Xbounded}}
Let $\Omega_e$ be the event defined in \eqref{def.Omegae}. Then for
every $\omega\in \Omega_e$ we may define
$z(t,\omega):=X(t,\omega)-Y(t,\omega)$ for $t\geq 0$ where $Y$ is
the solution of \eqref{eq.linsde}. Then $z(0)=X(0)$ and each sample
path of $z$ is in $C^1(0,\infty)$ with
\[
z'(t,\omega)=-f(z(t,\omega)+Y(t,\omega))+Y(t,\omega), \quad t>0.
\]
If $\theta$ obeys \eqref{eq.theta1} and \eqref{eq.theta2}, it
follows from part (B) of Theorem~\ref{theorem.Ybounded} that there
exists an a.s. event $\Omega_1$, defined by \eqref{def.Omega1}, such
that there is a finite, positive and deterministic $\overline{Y}$
satisfying \eqref{def.overY} i.e.
\[
\overline{Y}=\sup_{\omega\in \Omega_1} \limsup_{t\to\infty}
|Y(t,\omega)|.
\]
Let $\Omega_2=\Omega_1\cap \Omega_e$. Fix $\omega\in \Omega_2$. Then by Lemma~\ref{lemma.odelimsup}, with
$Y(\cdot,\omega)$ in the role of $p$, and $z(\cdot,\omega)$ in the role of $z$, we have that
\[
\limsup_{t\to\infty} |z(t,\omega)|\leq \max(x_+(\overline{Y}),x_-(\overline{Y})).
\]
Putting $X(\cdot,\omega)$ in the role of $x$ in Lemma~\ref{lemma.odelimsup}, we can infer from
Lemma~\ref{lemma.odelimsup} that for $\omega\in \Omega_2$
\[
\limsup_{t\to\infty} |X(t,\omega)|\leq \overline{x}(f,\overline{Y}).
\]
Since this estimate holds for all $\omega\in \Omega_2$, we have precisely \eqref{eq.limsupXuprboundfinal}, as required.

\subsection{Proof of Proposition~\ref{prop.overline}}
For a given $f$, $f^+$ and $f^-$ are non--decreasing functions. We
show first that $\lim_{x\to 0} f^+(x)=0$.

Since $f(x)\to\infty$ as $x\to\infty$, there exists $a>0$ such that
$f(x)\geq 1$ for all $x\geq a$. Let $\epsilon$ be any positive
number with $\epsilon<a$. Then, as $f$ is continuous and strictly
positive on $[\epsilon,a]$, it follows that there exists
$x_\epsilon\in [\epsilon,a]$ such that
$0<f(x_\epsilon)=\min_{\epsilon\leq y\leq a} f(y)$. Define
$\delta(\epsilon)=f(x_\epsilon)$. Then, if $0<x<\delta(\epsilon)$,
we have that $f^+(x)\leq \epsilon$. To justify this, suppose to the
contrary that $f^+(x')>\epsilon$ for some $x'\in
(0,\delta(\epsilon))$. Then $f^+(x')=\sup\{z>0:f(z)=x'\}>\epsilon$.
Now, for $f^+(x')=:z'> \epsilon$, we have $f(z')\geq
f(x_\epsilon)=\delta(\epsilon)$. However, by hypothesis
$\delta(\epsilon)>x'$, so $f(z')>x'$. However,
$z'=f^{+}(x')=\sup\{z>0:f(z)=x'\}$ implies that $f(z')=x'$, so we
have a contradiction. Therefore, for every $\epsilon\in (0,a)$ there
exists a $\delta=\delta(\epsilon)>0$ such that if
 $0<x<\delta(\epsilon)$, we have that $f^+(x)\leq \epsilon$. Thus, as $f^+$ is a non--negative function, and $\epsilon\in (0,a)$ is arbitrary, this is precisely $\lim_{x\to 0^+} f^+(x)=0$.
The proof that $\lim_{x\to 0} f^-(x)=0$ is similar.

Note that $\lim_{x\to\infty} f^+(x)=\lim_{x\to\infty} f^-(x)=\infty$
(by \eqref{h.finfty}) so it is clear that $y\mapsto
\overline{x}(f,y)$ is increasing, and moreover that
$\lim_{y\to\infty} \overline{x}(f,y)= \infty$.
Also, as  $\lim_{x\to 0^+} f^+(x)=\lim_{x\to 0} f^-(x)=0$, we have
that $\lim_{y\to 0} \overline{x}(f,y)= 0$,
which proves part (i).

To prove part (ii),
suppose first that there is $x>0$ such that ${f_1}^+(x)<{f_2}^+(x)$.
By definition, $f_1(z)>x$ for all $z>{f_1}^+(x)$. Since
${f_1}^+(x)<{f_2}^+(x)$, we have $f_1({f_2}^+(x))>x$. But
$f_2({f_2}^+(x))\geq f_1({f_2}^+(x))$ by \eqref{eq.f2f1}. Hence
$f_2({f_2}^+(x))>x$. But $f_2({f_2}^+(x))=x$, by definition, so we
have the contradiction $x>x$. Hence
\begin{equation} \label{eq.if1vsif2}
{f_1}^+(x)\geq {f_2}^+(x), \quad x>0.
\end{equation}
Suppose next there is $y<0$ such that $f_{1}^-(y)>f_{2}^-(y)$. By
definition, $f_1(z)<y$ for $z<f_{1}^-(y)$. Since
$f_2^-(y)<f_1^-(y)$, it follows that $f_1(f_2^-(y))<y$. By
\eqref{eq.f2f1}, we have $-f_2(u)\geq -f_1(u)$ for all $u<0$. Hence
with $u=f_2^-(y)$, we get $-f_2(f_2^-(y))\geq -f_1(f_2^-(y))>-y$.
But $f_2(f_2^-(y))=y$, by definition, so we have
$-y=-f_2(f_2^-(y))\geq -f_1(f_2^-(y))>-y$, a contradiction. Thus we
have $f_1^-(y)\leq f_2^-(y)$ for all $y<0$, or
\begin{equation} \label{eq.if1vsif22}
-f_1^-(y)\geq -f_2^-(y), \quad y<0.
\end{equation}
Therefore, it follows from \eqref{def.overlineX},
\eqref{eq.if1vsif2} and \eqref{eq.if1vsif22} that
\begin{align*}
\overline{x}(f_2,y)
&=2y+\max({f_2}^+(y),-f_2^-(-y))\\
&\leq 2y+
\max({f_1}^+(y),-{f_1}^-(-y))=\overline{x}(f_1,y),
\end{align*}
as required.


\begin{thebibliography}{10}

\bibitem{JAGBAR:2008}
 J.~A.~D.~Appleby, G. Berkolaiko and A. Rodkina,
On local stability for a nonlinear difference equation with a
non-hyperbolic equilibrium and fading stochastic perturbations,
\emph{J. Differ. Eqns Appl.}, \textbf{14} (2008), 923–-951.


\bibitem{AppCheRod:2010a}
J.~A.~D.~Appleby, J.~Cheng and A.~Rodkina, Characterisation of the
Asymptotic Behaviour of Scalar Linear Differential Equations with
Respect to a Fading Stochastic Perturbation, submitted, 2013.

\bibitem{JAJCAR:2010}
J.~A.~D.~Appleby, J.~Cheng and A.~Rodkina, The split-step
Euler--Maruyama method preserves asymptotic stability for simulated
annealing problems, \emph{Proc. Neural, Parallel \& Scientific
Computations} IV, 31--36, 2010. 

\bibitem{AppCheRod:dres}
J.~A.~D.~Appleby, J.~Cheng and A.~Rodkina, 
Characterisation of the asymptotic behaviour of scalar linear differential equations with respect to a fading stochastic perturbation, 
\emph{Disc. Continuous. Dynam. Systems, Supplement}, 79--90, 2011.

\bibitem{JAJGAR:2009}
J.~A.~D. Appleby, J.~G.~Gleeson and A.~Rodkina. On asymptotic
stability and instability with respect to a fading stochastic
perturbation, \emph{Applicable Analysis}, 88 (4), 579-–603, 2009.

\bibitem{JACKXMAR:2010}
J.~A.~D.~Appleby, C.~Kelly, X.~Mao, and A.~Rodkina, On the local
stability and instability of polynomial difference equations with
fading stochastic perturbations, \emph{Dynamics of Continuous,
Discrete and Impulsive Systems Series A: Mathematical Analysis},
\textbf{17} (2010), 401--430.

\bibitem{JAARMR:2009}
J.~A.~D.~Appleby,  M.~Riedle and A.~Rodkina. On Asymptotic Stability
of linear stochastic Volterra difference equations with respect to a
fading perturbation, \emph{Adv. Stud. Pure Math.}, 53, 271--282,
2009.

\bibitem{Chan:1989}
T.~Chan. On multi--dimensional annealing problems, \emph{Math. Proc.
Camb. Philos. Soc.}, 105, 177-–184, 1989.

\bibitem{ChanWill:1989}
T.~Chan and D.~Williams. An ``excursion'' approach to an annealing
problem, \emph{Math. Proc. Camb. Philos. Soc.}, 105, 169-–176, 1989.

\bibitem{K&S}
I. Karatzas and S.~E.~Shreve, ``Brownian motion and stochastic
calculus", $2^{nd}$ edition, Springer-Verlag, New York, 1991.

\bibitem{LipShir:1989}
R. Sh. Liptser and A. N. Shiryaev, Theory of Martingales, Kluwer
Academic Publishers, Dordrecht, 1989.

\bibitem{Mao1}
X.~Mao.
\newblock {\em Stochastic Differential Equations and their Applications}.
\newblock Horwood Publishing Limited, Chichester, 1997.

%
%
\end{thebibliography}
\end{document}